\documentclass[12pt,reqno]{amsart}\usepackage{etex}
\usepackage{amsmath, amssymb, amsthm, mathrsfs}
\usepackage{stmaryrd} 
\usepackage{color} 
\usepackage{mathtools} 
\usepackage{bbm} 
\usepackage{hyperref}
\usepackage{leftidx} 
\usepackage{mathabx} 
\usepackage{mathtools} 
\usepackage{rotating} 
\usepackage{soul}
\allowdisplaybreaks
\hypersetup{colorlinks,linkcolor=blue,urlcolor=cyan,citecolor=blue}

\usepackage{arydshln}
\makeatletter
\renewcommand*\env@matrix[1][*\c@MaxMatrixCols c]{%
  \hskip -\arraycolsep
  \let\@ifnextchar\new@ifnextchar
  \array{#1}}
\makeatother
\usepackage{pgf,tikz}
\usetikzlibrary{arrows, positioning, calc, chains}
\tikzset{
	ch/.style={circle,draw,on chain,inner sep=2pt},
	chj/.style={ch,join},
	every path/.style={shorten >=4pt,shorten <=4pt}
	}
\newcommand{\dnode}[2][chj]{%
	\node[#1,label={below:#2}] (#1) {};}
\newcommand{\dnodenj}[1]{%
	\dnode[ch]{#1}}
\newcommand{\dydots}{%
	\node[chj,draw=none,inner sep=1pt] {\dots};}
\numberwithin{equation}{subsection}

\newcommand{\bbH}{\mathbb{H}}
\newcommand{\bbU}{\mathbb{U}}
\newcommand{\bSjj}{\mathbb{S}^{\fc}_{n,d}}
\newcommand{\bSij}{\mathbb{S}^{\imath\jmath}_{\nn,d}}
\newcommand{\bSji}{\mathbb{S}^{\jmath\imath}_{\nn,d}}
\newcommand{\bSii}{\mathbb{S}^{\imath\imath}_{\eta,d}}
\newcommand{\bTjj}{\mathbb{T}^{\fc}_{n,d}}
\newcommand{\bTij}{\mathbb{T}^{\imath\jmath}_{\nn,d}}
\newcommand{\bTji}{\mathbb{T}^{\jmath\imath}_{\nn,d}}
\newcommand{\bTii}{\mathbb{T}^{\imath\imath}_{\eta,d}}
\newcommand{\etil}{\widetilde{e}}
\newcommand{\ftil}{\widetilde{f}}

\newcommand{\bA}[1]{    \begin{align}  #1  \end{align}}
\newcommand{\bAn}[1]{    \begin{align*}  #1  \end{align*}}

\newcommand{\bc}[1]{     \begin{cases}   #1   \end{cases}}
\newcommand{\bD}{\mbf{D}}
\newcommand{\bE}{\mbf E}
\newcommand{\be}{\mbf e}

\newcommand{\bF}{\mbf F}
\newcommand{\bbf}{\mbf f}
\newcommand{\bH}{\mbf H}
\newcommand{\bt}{\mbf t}

\newcommand{\bh}{\mbf h}

\newcommand{\bK}{\mbf K}
\newcommand{\bk}{\mbf k}

\newcommand{\Bp}[1]{\Big{(} #1\Big{)}}

\newcommand{\cR}{\mathbb{F}}

\newcommand{\D}{\mathscr{D}} 

\newcommand{\End}{\textup{End}}
\newcommand{\ep}{\epsilon}

\newcommand{\fc}{\mathfrak{c}}

\newcommand{\fgl}{\mathfrak{gl}}

\newcommand{\fsl}{\mathfrak{sl}}

\newcommand{\HH}{\mathbb{H}}
\newcommand{\Hom}{\textup{Hom}}

\newcommand{\inv}{^{-1}}

\newcommand{\jjw}{\mbox{$\jmath$\kern-1.3pt$\jmath$}}
\newcommand{\jiw}{\mbox{$\jmath$\kern-0.8pt$\imath$}}
\newcommand{\ijw}{\mbox{$\imath$\kern-1.8pt$\jmath$}}
\newcommand{\iiw}{\mbox{$\imath$\kern-1.0pt$\imath$}}
\newcommand{\nn}{\mathfrak{n}}

\newcommand{\ld}{\lambda}
\newcommand{\Ld}{\Lambda}

\newcommand{\mbf}{\mathbf}

\newcommand{\NN}{\mathbb{N}}

\newcommand{\otw}{\textup{otherwise}}

\newcommand{\QQ}{\mathbb{Q}}

\newcommand{\tif}{\textup{if }}

\newcommand{\UU}{\mathbb{U}}
\newcommand{\UUA}{\UU(\^{\fsl}_n)}
\newcommand{\UUAg}{\UU(\^{\fgl}_n)}
\newcommand{\UUC}{\mathbb{U}^{\mathfrak{c}}(\mathfrak{\widehat{sl}}_n)}
\newcommand{\UUCg}{\mathbb{U}^{\mathfrak{c}}(\mathfrak{\widehat{gl}}_n)}
\newcommand{\UUAij}{\UU(\^{\fsl}_{\nn})}
\newcommand{\UUAijg}{\UU(\^{\fgl}_{\nn})}
\newcommand{\UUAji}{\UU('\^{\fsl}_{\nn})}
\newcommand{\UUAjig}{\UU('\^{\fgl}_{\nn})}
\newcommand{\UUAii}{\UU(\^{\fsl}_{\eta})}
\newcommand{\UUAiig}{\UU(\^{\fgl}_{\eta})}
\newcommand{\UUij}{\mathbb{U}^{\imath\jmath}(\mathfrak{\widehat{sl}}_{\nn})}

\newcommand{\UUji}{\mathbb{U}^{\jmath\imath}(\mathfrak{\widehat{sl}}_{\nn})}

\newcommand{\UUii}{\mathbb{U}^{\imath\imath}(\mathfrak{\widehat{sl}}_{\eta})}

\newcommand{\VV}{\mathbb{V}}
\newcommand{\W}{W}

\newcommand{\ZZ}{\mathbb{Z}}

\renewcommand{\^}[1]{\widehat{#1}}

\renewcommand{\=}[1]{\overline{#1}}

\newcommand{\ji}{\jmath \imath}
\newcommand{\eji}{{\mbf e}}
\newcommand{\fji}{{\mbf f}}
\newcommand{\kji}{{\mbf k}}

\newcommand{\tji}{{\mbf t}}

\newcommand{\ii}{\imath \imath}

\newenvironment{enu}{\begin{enumerate}}{\end{enumerate}}
\usepackage{enumitem }

\newenvironment{eq}{\begin{equation}}{\end{equation}}

\theoremstyle{definition}
\newtheorem{Def}{Definition}[section] 

\newtheorem{rem}[Def]{Remark}

\theoremstyle{plain}
\newtheorem{prop}[Def]{Proposition}
\newtheorem{thm}[Def]{Theorem}
\newtheorem{lem}[Def]{Lemma}

\newtheorem{cor}[Def]{Corollary}

\newtheorem{theorem}[Def]{Theorem}
\newtheorem{lemma}[Def]{Lemma}

\title[Schur duality with three parameters]{Quantum Schur duality of affine type C with three parameters}
\author[Fan, Lai, Li, Luo, Wang and Watanabe]{Z. Fan,  C. Lai, Y. Li,  L. Luo, W. Wang and H. watanabe}
\address{School of science, Harbin Engineering University, Harbin 150001, China}
    \email{fanz@math.ksu.edu  (Fan)}
\address{
Department of Mathematics, University of Georgia, Athens, GA 30605}
\email{cjlai@uga.edu (Lai)}
\address{Department of Mathematics, University at Buffalo, The State University of New York,  Buffalo, NY 14260}
    \email{yiqiang@buffalo.edu (Li)}
\address{
    Department of mathematics, Shanghai Key Laboratory of Pure Mathematics and Mathematical Practice, East China Normal University, Shanghai 200241, China}
\email{lluo@math.ecnu.edu.cn (Luo)}
\address{Department of Mathematics, University of Virginia, Charlottesville, VA 22904}
\email{ww9c@virginia.edu (Wang)}
\address{Department of Mathematics, Tokyo Institute of Technology, 2-12-1 Oh-okayama, Meguro-ku, Tokyo 152-8550, Japan}
    \email{watanabe.h.at@m.titech.ac.jp (Watanabe)}

\begin{document}

\begin{abstract}
We establish a three-parameter Schur duality between the affine Hecke algebra of type C and a coideal subalgebra of quantum affine $\mathfrak{sl}_n$. At the equal parameter specializations, we obtain Schur dualities of types BCD.
\end{abstract}
\maketitle

\setcounter{tocdepth}{1}
\tableofcontents

\section{Introduction}

The classic Schur duality exhibits the fundamental interactions between representation theories of general linear Lie algebras and symmetric groups. A quantized Schur duality was obtained by Jimbo \cite{Jim86} between quantum groups and Hecke algebras of type A. There has also been various versions of affine type A Schur duality; cf. \cite{CP94, Gr99}.

In developing a Kazhdan-Lusztig theory of (super) type BCD, Bao and Wang  \cite{BW18} were led to a Schur type duality between Hecke algebra of type B and an $\imath$quantum group which is a coideal subalgebra of a quantum group of type A. There has been further development of such dualities which involve Hecke algebra of type B of unequal or two parameters; see  \cite{Bao17, BWW18}. We recall a coideal subalgebra $\mathbb{U}^\imath$ of a quantum group $\mathbb{U}$ together form a quantum symmetric pair $(\mathbb{U}, \mathbb{U}^\imath)$; see \cite{Le99, Ko14}.

The goal of this paper is to formulate and establish a Schur $(\UUC, \bbH)$-duality on $\mathbb{V}^{\otimes d}$, for $n\ge 2d+2$ (and three additional variants). Here $\mathbb{V}$ is an infinite-dimensional vector space with a basis parametrized by $\ZZ$, $\bbH$ denotes the Hecke algebra of affine type $C_d$ in three parameters, and $\UUC$ is  an affine $\imath$quantum group which is a coideal subalgebra of the affine type A quantum group $\UUA$. The actions of $\UUC$ and $\bbH$ on $\mathbb{V}^{\otimes d}$ are given by explicit formulas.

 It is well known (cf. \cite{Lu89, Ka09, VV11}) there is a 3-parameter Hecke algebra $\bbH$ of affine type C over $\QQ(q,q_0,q_1)$ which specializes to all  kinds of Hecke algebras of classical affine types.
Remarkably, in the general theory of quantum symmetric pairs  \cite{Le99, Ko14}, a coideal subalgebra of the quantum groups of affine type A allows different parameters. In our setting suitable choices of the parameters in the coideal subalgebras correspond to the 3 parameters of Hecke algebras of affine type C.

A geometric approach and a Hecke algebraic approach were systematically developed in \cite{FL3Wa, FL3Wb} (also see \cite{BKLW}) toward the realizations of coideal subalgebras of quantum groups of affine type A and constructions of their canonical bases. A Schur duality involving affine Hecke algebra of type C (of single parameter) was implicit in these papers and could be developed in those frameworks naturally. It is conceivable that there will be other type of Schur dualities (of single parameter) if one starts with different types of affine flag varieties or affine Hecke algebras, and it would take considerable work to set this up. Upon single parameter specializations, the 3-parameter duality here immediately leads to several dualities involving Hecke algebras of different affine types, which are expected to arise from geometric constructions using different types of flag varieties. In this way, the 3-parameter Schur duality in this paper could serve as a helpful guideline on the geometric and categorical realizations of various equal parameter Schur dualities of different types in the future.

In a very interesting work \cite{CGM14}, Chen, Guay and Ma considered a duality which is reminiscent to our but different in several aspects. They considered 2-parameter (instead of 3-parameter here) affine Hecke algebras, and their formulation uses {\em finite-dimensional} tensor representations. The coideal algebra therein used a different definition via reflection equations, and it is not known (though is expected) if it is isomorphic to some suitable specialization of the one used in this paper. It is interesting and should be possible to adapt our work to study finite-dimensional representations of the multi-parameter coideal algebras as well.

The paper is organized as follows. In Section~\ref{sec:aqa}, we define an infinite-dimensional tensor module $\mathbb{V}^{\otimes d}$ for the affine Hecke algebra $\bbH$.
The Schur $(\UUC, \bbH)$-duality is established in Section~\ref{sec:Schur}.
In Section~\ref{sec:var},  inspired by the considerations in \cite{FL3Wa}--\cite{FL3Wb}, we establish three additional variants of Schur duality: the $(\UUji, \bbH)$-duality, the $(\UUij, \bbH)$-duality, and the $(\UUii, \bbH)$-duality. Here $\UUji, \UUij, \UUii$ denote different coideal subalgebras in $\bbU(\widehat{\mathfrak{sl}}_{n-1}), \bbU(\widehat{\mathfrak{sl}}_{n-1}), \bbU(\widehat{\mathfrak{sl}}_{n-2})$, respectively.

\vspace{3mm}
\noindent {\bf Acknowledgements.}
We thank Huanchen Bao for his work and idea which inspired the multiparameter Schur duality here.
We thank East China Normal University and University of Virginia whose support and hospitality help to facilitate the work of this project.
Z.~Fan is partially supported by the NSF of China grant 11671108, the NSF of Heilongjiang Province grant LC2017001 and the Fundamental Research Funds for the central universities GK2110260131. 
L. Luo is supported by Science and Technology Commission of Shanghai Municipality grant 18dz2271000 and the NSF of China grant 11871214. W. Wang is partially supported by the NSF grant DMS-1702254. H. Watanabe is supported by JSPS KAKENHI Grant Number 17J00172. We thank the referee for a careful reading and helpful comments.

\section{Quantum algebras and Hecke algebras}
  \label{sec:aqa}

In this section we give a quick review on the quantum group of affine type A, its coideal subalgebra $\UUC$, and the Hecke algebra $\bbH$ of affine type C. We formulate the actions of $\UUC$ and $\bbH$ on the tensor space $\mathbb{V}^{\otimes d}$.

\subsection{Quantum group of affine type A}
\label{Quantum}

Let $q, q_0, q_1$ be indeterminates, and denote by $\cR$ the field
$$\cR = \QQ(q, q_0, q_1).$$
The quantum affine $\fgl_n$  is the associative algebra $\UUAg$ over $\cR$
generated by
\[
\bE_i, \bF_i
\quad
(0 \leq i \leq n-1)
\quad
\bD_a^{\pm1}
\quad
(0 \leq a \leq n-1)
\]
subject to the following relations for $0 \leq a,b \leq n-1$ and for $0 \leq i,j\leq n-1$:
\enu
\item $q$-Cartan relations:
\begin{align*}
&\bD_a\bD_b=\bD_b\bD_a,
\quad
\bD_a\bD_a\inv=1 = \bD_a\inv\bD_a,
\\
&\bD_a\bE_j\bD_a\inv=q^{\delta_{aj}-\delta_{a-1,j}}\bE_j,
\quad
\bD_a\bF_j\bD_a\inv=q^{-\delta_{aj}+\delta_{a-1,j}}\bF_j,
\\
&\bE_i\bF_j-\bF_j\bE_i
=\delta_{ij}\frac{\bK_i-\bK_i\inv}{q-q^{-1}}.
\end{align*}
(Here and below $\bK_i := \bD_{i}\bD_{i+1}\inv$ and $\bD_n=\bD_0$.)

\item $q$-Serre relations:
\begin{align*}
&\bE_i^2\bE_j+\bE_j\bE_i^2=(q+q^{-1})\bE_i\bE_j\bE_i,& \tif |i-j| \equiv 1,
\\
&\bF_i^2\bF_j+\bF_j\bF_i^2=(q+q^{-1})\bF_i\bF_j\bF_i,& \tif |i-j| \equiv 1,
\\
&\bE_i\bE_j=\bE_j\bE_i,
\quad
\bF_i\bF_j=\bF_j\bF_i,& \tif i\not\equiv j\pm1,
\end{align*}
\endenu
where $i\equiv j$ means $i\equiv j \pmod n$.
The quantum affine $\fsl_n$  is the $\cR$-subalgebra $\UUA$ of $\UUAg$ generated by $\bE_i, \bF_i, \bK_i^{\pm 1}$
$(0 \leq i  \leq n-1)$.

\begin{rem}
The algebra $\UUAg$ does not contain a ``Heisenberg subalgebra'' and it differs from $\UUA$ only on the finite Cartan subalgebra; it plays only an auxiliary role as it allows for simpler formulas. The algebra $\UUA$ has level 0 and is sometimes called the quantum loop algebra of $\mathfrak{sl}_n$.
\end{rem}

The comultiplication $\Delta$ on $\UUAg$ is given as follows:
\begin{equation*}
\Delta(\bE_i) = \bE_i \otimes \bK_i\inv + 1 \otimes \bE_i,
\quad
\Delta(\bF_i) = \bF_i \otimes 1 + \bK_i \otimes \bF_i,
\quad
\Delta(\bD_a) = \bD_a \otimes \bD_a.
\end{equation*}
Let $\mathbb{V}$ be the $\cR$-vector space with basis $\{v_j \mid j\in\mathbb{Z}\}$.
It has a natural module structure over $\UUAg$ (and hence over $\UUA$) as follows:
\eq\label{actionU}
\bE_{i}v_{j+1}=
\bc{v_{j}&\tif j\equiv i;
\\
0&\text{else},}
\quad
\bF_{i}v_{j}=\bc{
v_{j+1} &\tif j\equiv i;
\\
0 &\text{else},
}
\quad
\bD_{a}v_{j}=
\bc{q v_{j} &\tif j\equiv a;
\\
v_{j} &\text{else}.
}
\endeq


%
%
\subsection{An $\imath$quantum group}

From now on we take an integer $r\ge 1$, and let
\[
n=2r+2.
\]
Let $\UUCg$ be the associative algebra over $\cR$
generated by
\[
\be_i, \bbf_i\quad (0 \leq i \leq r),
\quad\bh_a^{\pm1}\quad (0 \leq a \leq r+1),
\]
subject to the following relations (in which $\bk_i := \bh_i \bh_{i+1}\inv$) for $0 \leq a,b \leq r+1, 0\leq i,j\leq r$:
\enu
\item $q$-Cartan relations:
\begin{align*}
&\bh_a\bh_b=\bh_b\bh_a,
\quad
\bh_a\bh_a\inv=1 = \bh_a\inv\bh_a,
\\
&\bh_a\be_j\bh_a\inv=\bc{
q^{2\delta_{0j}} \be_j &\tif a=0;
\\
q^{-2\delta_{rj}} \be_j &\tif a=r+1;
\\
q^{\delta_{aj}-\delta_{a-1,j}}\be_j &\otw,
}
\quad
\bh_a\bbf_j\bh_a\inv=\bc{
q^{-2\delta_{0j}} \bbf_j &\tif a=0;
\\
q^{2\delta_{rj}} \bbf_j &\tif a=r+1;
\\
q^{\delta_{a-1,j}-\delta_{aj}}\bbf_j &\otw,
}
\\
&\be_i\bbf_j-\bbf_j\be_i
=\delta_{i,j}\frac{\bk_{i}-\bk_{i}\inv}{q-q^{-1}}
\quad (i,j) \neq (0,0), (r,r).
\end{align*}
\item $q$-Serre relations:
\begin{align*}
&\be_i^2\be_j+\be_j\be_i^2=(q+q^{-1})\be_i\be_j\be_i,
\quad
\bbf_i^2\bbf_j+\bbf_j\bbf_i^2=(q+q^{-1})\bbf_i\bbf_j\bbf_i,
& \tif |i-j|=1,
\\
&\be_i\be_j=\be_j\be_i,
\quad
\bbf_i\bbf_j=\bbf_j\bbf_i,& \tif i\neq j\pm1,
\\
&\be_r^{2}\bbf_r + \bbf_r\be_r^{2}=(q+q^{-1})(\be_r\bbf_r\be_r-q^{2}q_0q_1^{-1}\be_r\bk_r-q^{-2}\be_r\bk_r^{-1}),
\\
&\bbf_r^{2}\be_r+\be_r\bbf_r^{2}=(q+q^{-1})(\bbf_r\be_r\bbf_r-q^{2}q_0q_1^{-1}\bk_r\bbf_r-q^{-2}\bk_r^{-1}\bbf_r),
\\
&\be_{0}^2\bbf_{0}+\bbf_{0}\be_{0}^2
=(q+q^{-1})(\be_{0}\bbf_{0}\be_{0}-q_1 q\be_{0}\bk_{0}-q_0^{-1}q^{-1}\be_{0}\bk_{0}\inv),
\\
&\bbf_{0}^2\be_{0}+\be_{0}\bbf_{0}^2
=(q+q^{-1})(\bbf_{0}\be_{0}\bbf_{0}-q_1q\bk_{0}\bbf_{0}-q_0^{-1}q^{-1}\bk_{0}\inv\bbf_{0}).
\end{align*}
\endenu
Let $\mathbb{U}^{\mathfrak{c}}(\mathfrak{\widehat{sl}}_n)$ be the subalgebra of $\mathbb{U}^{\mathfrak{c}}(\mathfrak{\widehat{gl}}_n)$ generated by $\be_i, \bbf_i, \bk_i^\pm (0 \leq i \leq r)$. Sometimes, $\mathbb{U}^{\mathfrak{c}}(\mathfrak{\widehat{sl}}_n)$ and $\mathbb{U}^{\mathfrak{c}}(\mathfrak{\widehat{gl}}_n)$ are called $\imath$quantum groups.

We adopt the following identification for all $i\in\ZZ$:
\[
\bE_i = \bE_{i+n},
\quad
\bF_i = \bF_{i+n},
\quad
\bD_i = \bD_{i+n},
\quad
\bK_i = \bK_{i+n}.
\]

\begin{prop}
 \label{prop:homjj}
There are injective $\cR$-algebra homomorphisms
$\jjw: \UUCg \rightarrow \UUAg$ and $\jjw: \UUC\rightarrow \UUA$ defined by
\bA{
\label{embeding1}
&\bh_a \mapsto \bD_a\bD_{-a}, \quad (0\leq a \leq r+1)
\\
\label{embeding2}
&\be_i\mapsto \bE_i+\bF_{-i-1}\bK_i^{-1},
\quad
\bbf_i\mapsto \bE_{-i-1} + \bF_i\bK_{-i-1}^{-1},
\quad (1\leq i \leq r-1)
\\
\label{embeding3}
&\be_0\mapsto\bE_0+q_0^{-1}\bF_{-1}\bK_{0}^{-1},
\quad
\bbf_0\mapsto \bE_{-1} + q_1q^{-1}\bF_0\bK_{-1}^{-1}.
\\
\label{embeding4}
&\be_r\mapsto \bE_r+q^{-1}\bF_{-r-1}\bK_r^{-1},
\quad
\bbf_r\mapsto \bE_{-r-1} + q_0q_1^{-1}\bF_r\bK_{-r-1}^{-1},
}
\end{prop}
It follows that $\bk_a \mapsto \bK_a\bK_{-a-1}^{-1},  \quad (0\leq a \leq r)$.
It turns out $(\UUAg, \UUCg)$ forms a quantum symmetric pair \`a la Letzter and Kolb.

\proof
Noting that the subalgebra of $\UUAg$ generated by the right-hand sides of \eqref{embeding1}-\eqref{embeding4} is a quantum symmetric pair coideal subalgebra (in the sense of \cite{Ko14}) associated with the affine Dynkin diagram and involution below, the proposition follows from {\cite[Theorem~ 7.8]{Ko14}}.
\begin{figure}[ht!]
\caption{Dynkin diagram of type $A^{(1)}_{2r+1}$ with involution of type $\jmath\jmath \equiv \fc$.}
 \label{figure:jj}
\[
\begin{tikzpicture}
\matrix [column sep={0.6cm}, row sep={0.5 cm,between origins}, nodes={draw = none,  inner sep = 3pt}]
{
	\node(U1) [draw, circle, fill=white, scale=0.6, label = 0] {};
	&\node(U2)[draw, circle, fill=white, scale=0.6, label =1] {};
	&\node(U3) {$\cdots$};
	&\node(U4)[draw, circle, fill=white, scale=0.6, label =$r-1$] {};
	&\node(U5)[draw, circle, fill=white, scale=0.6, label =$r$] {};
\\
	&&&&&
\\
	\node(L1) [draw, circle, fill=white, scale=0.6, label =below:$2r+1$] {};
	&\node(L2)[draw, circle, fill=white, scale=0.6, label =below:$2r$] {};
	&\node(L3) {$\cdots$};
	&\node(L4)[draw, circle, fill=white, scale=0.6, label =below:$r+2$] {};
	&\node(L5)[draw, circle, fill=white, scale=0.6, label =below:$r+1$] {};
\\
};
\begin{scope}
\draw (L1) -- node  {} (U1);
\draw (U1) -- node  {} (U2);
\draw (U2) -- node  {} (U3);
\draw (U3) -- node  {} (U4);
\draw (U4) -- node  {} (U5);
\draw (U5) -- node  {} (L5);
\draw (L1) -- node  {} (L2);
\draw (L2) -- node  {} (L3);
\draw (L3) -- node  {} (L4);
\draw (L4) -- node  {} (L5);
\draw (L1) edge [color = blue,<->, bend right, shorten >=4pt, shorten <=4pt] node  {} (U1);
\draw (L2) edge [color = blue,<->, bend right, shorten >=4pt, shorten <=4pt] node  {} (U2);
\draw (L4) edge [color = blue,<->, bend left, shorten >=4pt, shorten <=4pt] node  {} (U4);
\draw (L5) edge [color = blue,<->, bend left, shorten >=4pt, shorten <=4pt] node  {} (U5);
\end{scope}
\end{tikzpicture}
\]
\end{figure}
\endproof

Combining \eqref{actionU} and \eqref{embeding1}--\eqref{embeding4}, we obtain an explicit description of the $\mathbb{U}^{\mathfrak{c}}(\mathfrak{\widehat{gl}}_n)$-action on $\VV$ as below.
\begin{lemma}\label{UconV}
The vector space $\VV$ admits a $\mathbb{U}^{\mathfrak{c}}(\mathfrak{\widehat{gl}}_n)$-action as below. For $0 \leq a \leq r+1$ and for $i\neq 0,r$,
\bA{
&\bh_a(v_j)  =\bc{
q^2v_j &\tif a = 0,r+1; a\equiv j;
\\
qv_j &\tif a\neq 0,r+1; \pm a\equiv j;
\\
v_j &\otw,
}
\\
&\be_i(v_j)  =\bc{
v_{j-1}&\tif j\equiv i+1;
\\
v_{j+1}&\tif -j\equiv i+1;
\\
0 & \otw,
}
\quad
\bbf_i (v_j)  =\bc{
v_{j+1}&\tif j\equiv i;
\\
v_{j-1}&\tif -j\equiv i ;
\\
0 &\otw;
}
\\
&\be_{0}(v_j)  =\bc{
v_{j-1} &\tif  j \equiv1;
\\
q_0^{-1}v_{j+1} &\tif j \equiv -1;
\\
0 &\otw;}
\quad
\bbf_{0}(v_j) =\bc{
q_1v_{j+1}+v_{j-1}&\tif j\equiv 0;
\\
0 &\otw;
}
\\
&\be_{r}(v_j) =
\bc{
v_{j-1}+v_{j+1}&\tif j\equiv r+1;
\\
0 &\otw;
}
\quad
\bbf_{r}(v_j) =\bc{
q_0q_1^{-1}v_{j+1}&\tif j\equiv r;
\\
v_{j-1}&\tif j\equiv r+2;
\\
0 &\otw.
}
}
\end{lemma}

\subsection{Affine Hecke algebra in 3 parameters}

Let $W$ be the Weyl group of affine type $C_d$  
generated by $S = \{s_0, s_1, \ldots, s_d\}$ with the affine Dynkin diagram
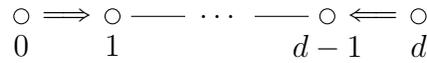
\[
\begin{tikzpicture}[start chain]
\dnode{$0$}
\dnodenj{$1$}
\dydots
\dnode{$d-1$}
\dnodenj{$d$}
\path (chain-1) -- node[anchor=mid] {\(\Longrightarrow\)} (chain-2);
\path (chain-4) -- node[anchor=mid] {\(\Longleftarrow\)} (chain-5);
\end{tikzpicture}
\]

Recall that $\VV$ is the natural representation of $\UUA$ with $\cR$-basis $\{v_i ~|~ i \in \ZZ\}$.
The tensor space $\VV^{\otimes d}$ then has an $\cR$-basis $\{M_f~|~ f\in \ZZ^d\}$, where
\begin{equation*}
M_f=v_{f_1}\otimes\cdots\otimes v_{f_d}\in \mathbb{V}^{\otimes d}
\quad
\textup{for}
\quad
f = (f_1, \ldots, f_d)\in \ZZ^d.
\end{equation*}
The group $W$ admits a natural right action on  $\ZZ^d$. Precisely, for $f = (f_1, \ldots, f_d) \in \ZZ^d$, we have
\begin{equation}
\label{eq:Wonf}
f\cdot s_i=\bc{
(f_1,\ldots,f_{i-1}, f_{i+1},f_i, f_{i+2},\ldots, f_d) & \tif i\neq 0,d;
\\
(-f_1,f_2,\ldots,f_d) &\tif i=0;
\\
(f_1,f_2,\ldots,f_{d-1},n-f_d) &\tif i=d.
}
\end{equation}

Let $\HH$ be the affine Hecke algebra of type $C_d$ with three parameters, that is, $\HH$ is an $\cR$-algebra generated by
\[
T_i
\quad
(0 \leq i \leq d-1),
\quad
X_a^{\pm1}
\quad
(1 \leq a \leq d),
\]
subject to the following relations, for $1\leq a,b \leq d$ and for $0\leq i,j,k \leq d-1$,
\enu
\item Toric relations:
\begin{equation*}
X_aX_a\inv = 1 = X_a\inv X_a,
\quad
X_aX_b = X_b X_a.
\end{equation*}
\item Hecke relations:
\begin{align}
 \label{eq:Hrel}
 \begin{split}
 &(T_0-q_0^{-1})(T_0+q_1)=0,
 \quad
(T_i-q^{-1})(T_i+q)=0
\quad
(i\neq0),
\\
&T_kT_{k-1}T_k=T_{k-1}T_kT_{k-1}
\quad(k\neq 0,1),
\\
&(T_0T_1)^2=(T_1T_0)^2,
\quad
T_iT_j=T_jT_i
\quad(|i-j|>1).
 \end{split}
\end{align}
\item Bernstein-Lusztig relations:
\begin{align}
 \label{T0X1}
&T_0X_1^{-1}T_0=q_0^{-1}q_1X_1+(q_0^{-1}q_1-1)T_0,\\
\label{TiXi}&T_iX_iT_i=X_{i+1} ~(i\neq0),
\quad
T_iX_j=X_jT_i ~(j\neq i,i+1).
\end{align}
\endenu

We remark that the Hecke algebra $\bbH$ of affine type $C_d$ in this paper can be matched with the version in \cite[Appendix~A]{VV11} with the following parameter correspondence: our $q \leftrightarrow \text{their } p$, our $q_0 \leftrightarrow \text{their }  q_1$, our $q_1 \leftrightarrow \text{their }  q_0$. Also see \cite{Ka09} in somewhat different notations.

The algebra $\HH$ contains a subalgebra $\HH_A$ generated by $T_1, \ldots, T_{d-1}, X_1^{\pm 1}, \ldots, X_d^{\pm 1}$, which is an affine Hecke algebra of type A.

We define $T_d \in \mathbb{H}$ by
\begin{equation}\label{eq:Td1}
T_d := q_0^{-1} X_d T_{d-1}^{-1} \cdots T_1^{-1} T_0^{-1} T_1^{-1} \cdots T_{d-1}^{-1}.
\end{equation}

\begin{lemma}
  \label{lem:braid}
The element $T_d \in \mathbb{H}$ satisfies the following relations:
\begin{enumerate}
\item $(T_d-q_1^{-1})(T_d+q_0^{-1})=0$.
\item $T_d T_i = T_i T_d$, for all $0 \leq i \leq d-2$.
\item $(T_{d-1}T_d)^2 = (T_dT_{d-1})^2$, if $d \geq 2$.
\end{enumerate}
\end{lemma}
\proof
These relations are verified by direct computations. Here we only present proofs for (1) and (3) while leaving the verification of (2) to the reader.

Thanks to \eqref{TiXi}, we have
\begin{align*}
  X_d T_{d-1}^{-1} \cdots T_1^{-1} T_0^{-1} = T_{d-1} \cdots T_1 (X_1 T_0^{-1}).
\end{align*}
Hence
\begin{equation}\label{Td2}
q_0T_d = T_{d-1} \cdots T_1(X_1 T_0^{-1})T_1^{-1}\cdots T_{d-1}^{-1}.
\end{equation}
It follows from \eqref{T0X1} that
$(X_1 T_0^{-1} - q_0q_1^{-1})(X_1 T_0^{-1} + 1) = 0$,
and thus $(q_0T_d - q_0q_1^{-1})(q_0T_d + 1) = 0$ by \eqref{Td2}.
Part (1) follows.

We compute
\begin{align*}
(T_{d-1}^{-1} T_d\inv)^2
&= q_0^2 T_{d-2} \ldots T_1T_0T_1 \ldots T_{d-1} X_d^{-1} T_{d-2} \ldots T_1T_0T_1 \ldots T_{d-1} X_d^{-1}
\\
&= q_0^2 T_{d-2} \ldots T_1T_0T_1 \ldots T_{d-1} T_{d-2} \ldots T_1T_0T_1 \ldots T_{d-2} T_{d-1}^{-1} X_{d-1}^{-1} X_d^{-1},
\\
(T_d\inv T_{d-1}^{-1})^2
&= q_0^2 T_{d-1} \ldots T_1T_0T_1 \ldots T_{d-1} X_d^{-1} T_{d-2} \ldots T_1T_0T_1 \ldots T_{d-1} X_d^{-1} T_{d-1}^{-1}
\\
& = q_0^2 T_{d-1} \ldots T_1T_0T_1 \ldots T_{d-1} T_{d-2} \ldots T_1T_0T_1 \ldots T_{d-2} T_{d-1}^{-1} X_{d-1}^{-1} X_d^{-1} T_{d-1}^{-1}
\\
&  = q_0^2 T_{d-1} \ldots T_1T_0T_1 \ldots T_{d-1} T_{d-2} \ldots T_1T_0T_1 \ldots T_{d-2} T_{d-1}^{-2} X_{d-1}^{-1} X_d^{-1}.
\end{align*}
Thus, to show the identity in (3) it suffices to show that
\eq\label{eq:TT}
(T_{d-2} \ldots T_0 \ldots T_{d-1}  \ldots T_0 \ldots T_{d-2}) T_{d-1}
= T_{d-1} (T_{d-2} \ldots T_0 \ldots T_{d-1}  \ldots T_0 \ldots T_{d-2}).
\endeq
Let $\alpha_0 = -2\ep_1, \alpha_{i} = \ep_i - \ep_{i+1}$ for $1\leq i\leq d-1$. The highest root in the finite type C Weyl group is $\theta = \alpha_0 + 2\alpha_1 + ... + 2\alpha_{d-2} + \alpha_{d-1} = -\ep_{d-1} - \ep_d$ and hence
\eq\label{eq:TT2}
 s_\theta s_{d-1} = s_{d-1} s_\theta.
\endeq
Therefore, \eqref{eq:TT} follows by \eqref{eq:TT2} and noting that
$s_{d-2} \ldots s_0 \ldots s_{d-1} \ldots s_0 \ldots s_{d-2}$ is a reduced expression of $s_\theta$.
\endproof

It follows by \eqref{TiXi}--\eqref{eq:Td1} that the algebra $\bbH$ is generated by $T_0, T_1, \ldots, T_d$. For any $w\in W$ with a reduced form $w=s_{i_1}\cdots s_{i_l}$, set
\begin{equation}\label{Tw}
T_w:=T_{i_1}\cdots T_{i_l}
\end{equation}
and
\begin{equation}\label{qw}
q_w:=q_{s_{i_1}}\cdots q_{s_{i_l}} \quad\mbox{where}\quad
q_{s_i} := \bc{
q_1 &\tif i = 0;
\\
q &\tif i \neq 0,d;
\\
q_0^{-1} &\tif i = d.}
\end{equation}
It follows by the braid relations in \eqref{eq:Hrel} and Lemma~\ref{lem:braid} that $T_w$ is independent of the choice of the reduced form of $w$.  Since the $q_{s_i}$'s satisfy the same braid relations, $q_w$ is uniquely determined by $w$, too.

\subsection{A tensor module for Hecke algebra}

We first recall a well-known action of the Hecke algebra $\HH_A$ of affine type A on $\VV^{\otimes d}$; see
\cite{KMS}. We introduce linear operators $z_1, \ldots, z_d$ which act on $\VV^{\otimes d}$ (from the right) as below:
\begin{equation*}
M_fz_i = M_{(f_1, \ldots, f_{i-1}, f_i + n, f_{i+1}, \ldots, f_d)}.
\end{equation*}
Since each $f_i\in\ZZ$ has a unique expression
$
f_i = \=f_i + c_i n,
$
for some $c_i\in\ZZ$ such that $-r\leq \=f_i \leq r+1$,
each basis element $M_f$ has a unique expression
\begin{equation*}
M_f= M_{\=f}Z_f,
\quad
Z_f = z_1^{c_1}\ldots z_d^{c_d}.
\end{equation*}

Recall the right $W$-action on $\ZZ^d$ in \eqref{eq:Wonf}.
Following \cite[(32)]{KMS}, the action of $\HH_A$ is given by, for $1\leq i \leq d-1$ and $1\leq a \leq d$,
\bA{
\label{Ti action}&M_fT_i=
\bc{
    M_{f\cdot s_i}
    +(q^{-1}-q)M_{\={f}}P^{(i)}_+(Z_f)
    & \tif \=f_{i+1}>\=f_i;
\\
    q^{-1}M_{\=f} Z_{f\cdot s_i}
    +(q^{-1}-q)M_{\={f}}P^{(i)}_+(Z_f)
    &\tif \=f_{i+1}=\=f_i;
\\
    M_{f\cdot s_i}
    +(q^{-1}-q)M_{\={f}}P^{(i)}_-(Z_f)
    & \tif \=f_{i+1}<\=f_i,
}
\\
\label{actionX}
&M_fX_a = M_f z_a\inv.
}
Here $P^{(i)}_\pm$ are operators given by
\eq
P^{(i)}_-(Z_f) =
\dfrac{z_{i+1}(Z_{f\cdot s_i}) - z_iZ_f}
{z_{i+1} - z_i},
\quad
P^{(i)}_+(Z_f) = \dfrac{z_i(Z_{f \cdot s_i}-Z_f)}
{z_{i+1} - z_i}.
\endeq

Now we shall enhance the action of $\HH_A$ on $\VV^{\otimes d}$ to an action of the Hecke algebra $\bbH$ of affine type C in 3 parameters.
For convenience we denote the basis elements of $\mathbb{V}$ by
\begin{equation*}
v_iz^j:=v_{i+nj}, \quad( -r\leq i\leq r+1, j\in\mathbb{Z}).
\end{equation*}
Define
\eq\label{T0 action}
M_fT_0= (v_{f_1}T_0) \otimes v_{f_2}\otimes \ldots \otimes v_{f_d},
\endeq
where $v_{f_1}T_0$ is given by
(below we assume $f_1=k+nj$, for $-r\leq k\leq r+1$):
\eq\label{T0onV}
\bc{
q_0^{-1}q_1v_{-k}z^{-j}
+(q_1-q_0^{-1})\sum_{l=1}^{j}v_kz^{j-2l}+(q_0^{-1}q_1-1)\sum_{l=1}^{j}v_kz^{j+1-2l}
&\tif k=r+1, j\geq 0;
\\
v_{-k}z^{-j}
+(q_0^{-1}-q_1)\sum_{l=1}^{-j}v_k z^{-j-2l}+(1-q_0^{-1}q_1)\sum_{l=2}^{-j}v_k z^{-j+1-2l}
&\tif  k=r+1, j<0;
\\
v_{-k}z^{-j}
+(q_1-q_0^{-1})\sum_{l=1}^{j}v_kz^{j-2l}+(q_0^{-1}q_1-1)\sum_{l=1}^{j}v_kz^{j+1-2l}
&\tif 0<k\leq r, j\geq 0;
\\
v_{-k}z^{-j}
+(q_0^{-1}-q_1)\sum_{l=1}^{-j}v_k z^{-j-2l}+(1-q_0^{-1}q_1)\sum_{l=1}^{-j}v_k z^{-j+1-2l}
&\tif  0<k\leq r, j<0;
\\
q_0^{-1}v_{0}z^{-j}
+(q_1-q_0^{-1})\sum_{l=1}^{j}v_0z^{j-2l}+(q_0^{-1}q_1-1)\sum_{l=1}^{j}v_0z^{j+1-2l}
&\tif k=0, j\geq 0;
\\
q_1 v_{0}z^{-j}
+(q_0^{-1}-q_1)
\sum_{l=0}^{-j}v_0 z^{-j-2l}+(1-q_0^{-1}q_1)\sum_{l=1}^{-j}v_0 z^{-j+1-2l}
&\tif k=0, j<0;
\\
q_0^{-1}q_1v_{-k}z^{-j}+(q_1-q_0^{-1})\sum_{l=1}^{j-1}v_kz^{j-2l}+(q_0^{-1}q_1-1)\sum_{l=1}^{j}v_kz^{j+1-2l}
&\tif -r\leq k<0, j> 0;
\\
q_0^{-1}q_1v_{-k}z^{-j}
+(q_0^{-1}-q_1)\sum_{l=0}^{-j}v_k z^{-j-2l}+(1-q_0^{-1}q_1)\sum_{l=1}^{-j}v_k z^{-j+1-2l}
&\tif -r\leq k<0, j\leq0.
}
\endeq

The formula \eqref{T0onV} above is obtained as follows. We first define the action of $T_0$ on $\{v_k|-r\leq k\leq r+1\}$ in $\mathbb{V}$, and then extend the action to all the basis vectors by the relation \eqref{T0X1}.

\begin{prop}
The formulas in \eqref{Ti action}--\eqref{T0onV} define an action of $\HH$ on $\mathbb{V}^{\otimes d}$.
\end{prop}
\proof
It suffices to check the Hecke relation \eqref{eq:Hrel} and the Bernstein-Lusztig relation \eqref{T0X1} for $T_0$. This follows by a direct computation, and here we only present the borderline cases in \eqref{T0onV}. It is useful to give the following formulas for the borderline cases in \eqref{T0onV}, for $1\leq k \leq r$,
\begin{align*}
&v_{r+1}T_0 = q_0^{-1}q_1v_{-r-1}
,
\\
&v_{-r-1}T_0 = v_{r+1}
+(q_0^{-1}-q_1)v_{-r-1}
,
\\
&v_k T_0 = v_{-k}
,
\\
&v_{k-n}T_0 = v_{n-k}
+(q_0^{-1}-q_1)v_{k-n}
+(1-q_0^{-1}q_1) v_k
,
\\
&v_0T_0 = q_0^{-1}v_{0}
,
\\
&v_{-n}T_0 = q_0\inv v_{n}
+(q_0^{-1}-q_1)
v_{-n}
+(1-q_0^{-1}q_1)v_0
,
\\
&v_{-k}T_0 = q_0^{-1}q_1v_{k}
+(q_0^{-1}-q_1) v_{-k}
,
\\
&v_{n-k} T_0 = q_0^{-1}q_1v_{k-n}
+(q_0^{-1}q_1-1)v_{-k}
.
\end{align*}
We start with checking  \eqref{eq:Hrel} for these cases as follows:
\begin{align*}
&v_{r+1} (T_0 - q_0\inv) (T_0 + q_1)
= q_0\inv q_1 (v_{r+1} + q_0\inv v_{-r-1}) - q_0\inv ( q_0\inv q_1 v_{-r-1}+q_1 v_{r+1})
= 0,
\\
&v_{-r-1} (T_0 - q_0\inv) (T_0 + q_1)
= q_0\inv q_1 v_{-r-1} + q_1 v_{r+1} - q_1 (  v_{r+1}+ q_0\inv  v_{-r-1})
= 0,
\\
&v_{k} (T_0 - q_0\inv) (T_0 + q_1)
 = (q_0\inv q_1 v_k + q_0\inv v_{-k})-q_0\inv( v_{-k} +q_1 v_{k}) = 0,
\\
&v_{-k} (T_0 - q_0\inv) (T_0 + q_1)
= q_0\inv q_1( v_{-k} + q_1 v_{k})-q_1( q_0\inv q_1 v_{k} +q_0\inv v_{-k}) = 0,
\\
&v_{0} (T_0 - q_0\inv) (T_0 + q_1)
= 0,
\\
&v_{-n} (T_0 - q_0\inv) (T_0 + q_1)
 = q_0\inv(q_1 v_{-n} + (q_0\inv q_1 - 1) v_0 + q_1 v_n)
 \\
 &+(1- q_0\inv q_1) (q_0\inv + q_1) v_0
 - q_1 ( q_0\inv v_{n} +(1-q_0\inv q_1) v_0 +q_0\inv v_{-n} )
 = 0,
 \\
&v_{k-n} (T_0 - q_0\inv) (T_0 + q_1)
= q_0\inv q_1 v_{k-n} + (q_0\inv q_1 -1) v_{-k}+q_1 v_{n-k}
\\
&- q_1 (v_{n-k} + q_0\inv v_{k-n} + (1-q_0\inv q_1) v_k)
+(1- q_0\inv q_1) (v_{-k} + q_1 v_k) = 0,
 \\
&v_{n-k} (T_0 - q_0\inv) (T_0 + q_1)
= q_0\inv q_1 (v_{n-k} + q_0\inv v_{k-n} + (1-q_0\inv q_1) v_k)
\\
&+(q_0\inv q_1 -1)(q_0\inv q_1 v_k + q_0\inv v_{-k})
- q_0\inv (q_0\inv q_1 v_{k-n} + (q_0\inv q_1 -1) v_{-k}+q_1 v_{n-k}) = 0.
\end{align*}
The Bernstein-Lusztig relation \eqref{T0X1} for the extremal cases follow from the following computation:
\begin{align*}
&v_{r+1}T_0 X_1\inv T_0 = (q_0^{-1}q_1)^2v_{-r-1},
\\
&v_{-r-1}T_0 X_1\inv T_0 =
q_0\inv q_1 v_{-3(r+1)}
+(q_0\inv q_1 - 1) ( v_{r+1} + (q_0\inv - q_1) v_{-r-1}),
\\
&v_k T_0 X_1\inv T_0 = q_0^{-1}q_1v_{k-n}
+(q_0^{-1}q_1-1)v_{-k},
\\
&v_{-k}T_0 X_1\inv T_0 = q_0\inv q_1 v_{-k-n}
+(q_0\inv q_1 -1)(q_0\inv q_1 v_k - (q_0^{-1} - q_1) v_{-k}),
\\
&v_0T_0 X_1\inv T_0 = q_0^{-1}(q_1 v_{-n} + (q_0\inv q_1 -1 ) v_0),
\\
&v_{-n}T_0 X_1\inv T_0
=q_0\inv q_1 v_{-2n} + (q_0\inv q_1 -1) (q_0 \inv v_n + (1-q_0\inv q_1) v_0 + (q_0\inv - q_1) v_{-n}),
\\
&v_{k-n}T_0 X_1\inv T_0 =
q_0\inv q_1 v_{k-2n} + (q_0\inv q_1 -1) (v_{n-k}+(q_0^{-1}-q_1)v_{k-n}+(1-q_0^{-1}q_1) v_k),
\\
&v_{n-k} T_0 X_1\inv T_0
= q_0^{-1}q_1v_{-k}+(q_0^{-1}q_1-1) (q_0^{-1}q_1v_{k-n}+(q_0^{-1}q_1-1)v_{-k}).
\end{align*}
The proposition is proved.
\endproof


The action of $T_i$ ($0 \leq i \leq d-1$) on the set $\{M_f~|~ 0\leq f_1\leq f_2\leq\cdots\leq f_d\leq r+1\}$ behaves nicely as below:
\bA{
\label{FD1}&M_fT_i=\bc{
q^{-1}M_f
&\tif 0\leq f_i=f_{i+1}\leq r+1,
\\
M_{f\cdot s_i}
&\tif 0\leq f_i <f_{i+1}\leq r+1,}
\quad (i\neq0)
\\
\label{FD2}&M_f T_0 = \bc{
q_0^{-1}M_{f}
& \tif f_1=0,
\\
M_{f\cdot s_0}
& \tif
0<f_1 < r+1;
\\
q_0\inv q_1 M_{f\cdot s_0}
& \tif
f_1 = r+1.
} }
Combining \eqref{FD1}--\eqref{FD2} with \eqref{actionX}, we obtain the following.

\begin{cor}\label{FD}
The tensor space $\mathbb{V}^{\otimes d}$ is generated by $\{M_f~|~0\leq f_1\leq f_2\leq\cdots\leq f_d\leq r+1\}$ as an $\mathbb{H}$-module.
\end{cor}

\section{Schur duality in three parameters}\label{sec:Schur}

In this section, we establish the Schur $(\UUC, \bbH)$-duality on $\mathbb{V}^{\otimes d}$. To that end, we study the structures of the affine Schur algebra.

\subsection{Affine Schur algebras}

From now on, we fix
\[
r, d\in\ZZ \quad \text{ such that } \quad
r\geq d\geq 1.
\]
Recall $n=2r+2$.
Let $\NN = \{ 0,1,2,\ldots \}$.
Denote the set of (weak) compositions of $d$ into $r+2$ parts
by
\eq \label{def:Ld}
\Ld_{n,d} := \Big \{ \ld = (\ld_0, \ld_1, \ldots, \ld_{r+1})\in\NN^{r+2} ~\big |~ \sum_{i=0}^{r+1} \ld_i = d \Big \}.
\endeq
For  $\ld \in \Ld_{n,d}$, let $W_\ld$ be the parabolic (finite) subgroup of $W$
generated by $S\backslash\{s_{\ld_0}, s_{\ld_{0,1}},\ldots, s_{\ld_{0,r}}\}$,
where $\ld_{0,i} =\ld_0 + \ld_1 + \ldots + \ld_i$ for $1 \leq i \leq r$;
note 
$\ld_{0,r} =d -\ld_{r+1}$.

We note that the element
\begin{equation*}
\omega:=(0, \underbrace{1,\ldots,1}_d, \underbrace{0,\ldots,0}_{r-d},0)\in\Ld_{n,d}
\end{equation*}
makes sense under the assumption $r\geq d$.

Recall $T_w$ in \eqref{Tw} and $q_w$ in \eqref{qw}.
For any finite subset $X \subset W$ and for $\lambda \in \Lambda_{n,d}$, set
$$T_X := \sum_{w \in X} q_w^{-1}T_w \quad\mbox{and}\quad x_{\lambda} := T_{W_{\lambda}}.$$

\begin{lemma}\label{xlm}
For $\lambda \in \Lambda_{n,d}$ and for $i \in \{ 0,1,\ldots,d \} \setminus \{ \lambda_0,\lambda_{0,1},\ldots,\lambda_{0,r} \}$, we have
\[
x_{\lambda} T_i = \bc{
q_0^{-1}x_{\lambda} &\tif i = 0;
\\
q^{-1}x_{\lambda} &\tif i \neq 0,d;
\\
q_1^{-1} x_{\lambda} & \tif i = d.
}
\]
\end{lemma}
\proof
Let us write $x_\lambda = \sum_{\substack{w \in W_\lambda \\ ws_i < w}}q_w^{-1}(T_w + q_{s_i} T_{ws_i})$. Then
\[
(T_w + q_{s_i}T_{ws_i})T_i = T_{ws_i}(T_i + q_{s_i})T_i = p_{s_i} T_{ws_i}(T_i + q_{s_i}) = p_{s_i}(T_w + q_{s_i}T_{ws_i}),
\]
where
\[
p_{s_i} := \bc{
q_0^{-1} &\tif i = 0;
\\
q^{-1} &\tif i \neq 0,d;
\\
q_1^{-1} &\tif i = d.}
\]
The lemma follows.
\endproof

The affine Schur algebra $\bSjj$ of 3-parameter is defined as the following $\cR$-algebra
$$\bSjj := \End_{\bbH}(\oplus_{\lambda \in \Lambda_{n,d}} x_{\lambda}\bbH) =\bigoplus_{\lambda,\mu \in \Lambda_{n,d}} \Hom_{\bbH}(x_{\mu}\bbH,x_{\lambda}\bbH).$$

Denote by $\ell(g)$ the length of $g\in W$. Let
\eq  \label{eq:Dmin}
\D_\ld := \big\{g\in \W ~|~ \ell(wg) = \ell(w) + \ell(g), \forall  w\in \W_\ld \big\}.
\endeq
Then $\D_\ld$ (respectively, $\D_\ld^{-1}$) is the set of minimal length right (respectively, left)
coset representatives of $\W_\ld$ in $\W$.
Denote by
\eq  \label{eq:Dmin2}
\D_{\ld\mu} = \D_\ld \cap \D_\mu^{-1}
\endeq
the set of minimal length double coset representatives
for $W_\ld \backslash W /W_\mu$.

For $\lambda,\mu\in \Ld_{n,d}$ and $g\in \D_{\ld\mu}$, define $\phi^g_{\lambda,\mu} \in \bSjj$ by
$$\phi^g_{\lambda,\mu}(x_\nu)=\delta_{\mu,\nu}T_{W_{\lambda} g W_{\mu}}, \quad \forall \nu\in\Ld_{n,d}.$$
It is straightforward to show that $\{\phi^g_{\lambda,\mu} \mid \lambda,\mu \in \Lambda_{n,d}, g \in \D_{\lambda,\mu} \}$ form an $\cR$-basis of $\bSjj$ (cf., e.g., \cite{DJ89, DDF12, FL3Wb}).

Define the right $\bbH$-module
\[
\bTjj := \bigoplus_{\lambda \in \Lambda_{n,d}} x_{\lambda}\bbH.
\]
Thanks to Corollary~\ref{FD} and Lemma \ref{xlm}, we have the following.
\begin{lemma}
  \label{lem:iso}
There exists a unique $\bbH$-module isomorphism
$\kappa: \bTjj \longrightarrow \mathbb{V}^{\otimes d}$ which sends
\begin{equation}
x_\lambda \mapsto M_{\lambda}:=M_{(0^{\lambda_0}, \ldots, r+1^{\lambda_{r+1}})} = v_0^{\otimes \lambda_0} \otimes \cdots \otimes v_{r+1}^{\otimes \lambda_{r+1}} \in \mathbb{V}^{\otimes d},\quad \forall \lambda \in \Lambda_{n,d}.
\end{equation}
This induces an algebra isomorphism $\bSjj \simeq \End_{\bbH}(\mathbb{V}^{\otimes d})$.
\end{lemma}
%
%
\subsection{The $\jmath\jmath$-Schur duality}

\begin{prop}\label{commute}
The actions of $\UUC$ and $\HH$ on $\mathbb{V}^{\otimes d}$ commute.
\end{prop}
\proof
It is known that the actions of $\UUA$ and $\HH_A$ on $\mathbb{V}^{\otimes d}$ commute.
It remains to check that the $T_0$-action commutes with the $\UUC$-action, and it suffices to check the special case $d=1$.

It follows from a direct computation (using Lemma \ref{UconV} and \eqref{T0onV}) that the $T_0$-action commutes with the actions of all generators of $\UUC$.
The calculation is simple except for $\be_0, \be_r$ and $\bbf_0, \bbf_r$, which are complicated but similar --
here we only provide a verification for $(\be_0 v)T_0 = \be_0( vT_0)$ and $(\be_r v)T_0 = \be_r( vT_0)$ for $v =  v_kz^j\in \VV$.

(1)  We claim that $(\be_0 v)T_0 = \be_0( vT_0)$.

Indeed, if $k\neq\pm1$, then
$(\be_0 v_kz^j)T_0=0=\be_0(v_kz^j T_0)$.
There are four cases remaining. If $k=1$ and $j\geq 0$, we have
\bAn{
(\be_0 v_1z^j)T_0&=q_0^{-1}v_{0}z^{-j}+(q_1-q_0^{-1})\sum_{l=1}^{j}v_0 z^{j-2l}+(q_0^{-1}q_1-1)\sum_{l=1}^jv_0z^{j+1-2l}
\\
&=\be_0\Bp{
v_{-1}z^{-j}+(q_1-q_0^{-1})\sum_{l=1}^{j}v_1 z^{j-2l}+(q_0^{-1}q_1-1)\sum_{l=1}^jv_1z^{j+1-2l}
}\\
&=\be_0(v_1z^j T_0).
}
If $k=1$ and $j<0$, we obtain
\bAn{
(\be_0 v_1z^j)T_0&=q_1v_{0}z^{-j}+(q_0^{-1}-q_1)\sum_{l=0}^{-j}v_0 z^{-j-2l}+(1-q_0^{-1}q_1)\sum_{l=1}^{-j}v_0z^{-j+1-2l}
\\
&=\be_0\Bp{
v_{-1}z^{-j}+(q_0^{-1}-q_1)\sum_{l=1}^{-j}v_1 z^{-j-2l}++(1-q_0^{-1}q_1)\sum_{l=1}^{-j}v_1z^{-j+1-2l}
}\\
&=\be_0(v_1z^j T_0).
}
For $k = -1$ and $j>0$, we have
\bAn{
(\be_0 v_{-1}z^j)T_0&=q_0^{-2}v_{0}z^{-j}+q_0^{-1}(q_1-q_0^{-1})\sum_{l=1}^{j}v_0 z^{j-2l}+q_0^{-1}(q_0^{-1}q_1-1)\sum_{l=1}^jv_0z^{j+1-2l}
\\
&=\be_0\Bp{
q_0^{-1}q_1v_{1}z^{-j}+(q_1-q_0^{-1})\sum_{l=1}^{j-1}v_{-1} z^{j-2l}+(q_0^{-1}q_1-1)\sum_{l=1}^jv_{-1}z^{j+1-2l}
}\\
&=
\be_0(v_{-1}z^j T_0).
}
Finally for  $k = -1$ and $j\leq 0$, we have
\bAn{
(\be_0 v_{-1}z^j)T_0&=q_0^{-1}q_1v_{0}z^{-j}+q_0^{-1}(q_0^{-1}-q_1)\sum_{l=0}^{-j}v_0 z^{-j-2l}+q_0^{-1}(1-q_0^{-1}q_1)\sum_{l=1}^{-j}v_0z^{-j+1-2l}
\\
&=\be_0 \Bp{
q_0^{-1}q_1v_{1}z^{-j}+(q_0^{-1}-q_1)\sum_{l=0}^{-j}v_{-1} z^{-j-2l}+(1-q_0^{-1}q_1)\sum_{l=1}^{-j}v_{-1}z^{-j+1-2l}
}\\
&=\be_0(v_{-1}z^j T_0).
}

(2) We claim that $(\be_r v)T_0=\be_r(v T_0)$.

Indeed, if $k\neq r+1$, then
$(\be_r v_kz^j)T_0=0=\be_r(v_kz^j T_0)$. There are two cases remaining. If $k=r+1$ and $j\geq 0$, we have
\bAn{
(\be_r v_{r+1}z^j)T_0&=(v_{r}z^j+v_{-r}z^{j+1})T_0\\
&=v_{-r}z^{-j}
+(q_1-q_0^{-1})\sum_{l=1}^{j}v_rz^{j-2l}+(q_0^{-1}q_1-1)\sum_{l=1}^{j}v_rz^{j+1-2l}\\
&\quad+q_0^{-1}q_1v_{r}z^{-j-1}+
(q_1-q_0^{-1})\sum_{l=1}^{j}v_{-r}z^{j+1-2l}+(q_0^{-1}q_1-1)\sum_{l=1}^{j+1}v_{-r}z^{j+2-2l}\\
&=q_0^{-1}q_1v_{-r}z^{-j}
+(q_1-q_0^{-1})\sum_{l=1}^{j}v_rz^{j-2l}+(q_0^{-1}q_1-1)\sum_{l=1}^{j}v_rz^{j+1-2l}\\
&\quad+q_0^{-1}q_1v_{-r-2}z^{-j}+
(q_1-q_0^{-1})\sum_{l=1}^{j}v_{r+2}z^{j-2l}+(q_0^{-1}q_1-1)\sum_{l=1}^{j}v_{r+2}z^{j+1-2l}\\
&=\be_r(q_0^{-1}q_1v_{-r-1}z^{-j}
+(q_1-q_0^{-1})\sum_{l=1}^{j}v_{r+1}z^{j-2l}+(q_0^{-1}q_1-1)\sum_{l=1}^{j}v_{r+1}z^{j+1-2l})\\
&=\be_r(v_{r+1}z^j T_0).
}
If $k=r+1$ and $j<0$, we have
\bAn{
(\be_r v_{r+1}z^j)T_0&=(v_{r}z^j+v_{-r}z^{j+1})T_0\\
&=v_{-r}z^{-j}
+(q_0^{-1}-q_1)\sum_{l=1}^{-j}v_r z^{-j-2l}+(1-q_0^{-1}q_1)\sum_{l=1}^{-j}v_r z^{-j+1-2l}\\
&\quad+q_0^{-1}q_1v_{r}z^{-j-1}
+(q_0^{-1}-q_1)\sum_{l=0}^{-j-1}v_{-r} z^{-j-1-2l}+(1-q_0^{-1}q_1)\sum_{l=1}^{-j-1}v_{-r} z^{-j-2l}\\
&=v_{-r}z^{-j}
+(q_0^{-1}-q_1)\sum_{l=1}^{-j}v_r z^{-j-2l}+(1-q_0^{-1}q_1)\sum_{l=2}^{-j}v_r z^{-j+1-2l}\\
&\quad +v_{-r-2}z^{-j}
+(q_0^{-1}-q_1)\sum_{l=1}^{-j}v_{r+2} z^{-j-2l}+(1-q_0^{-1}q_1)\sum_{l=2}^{-j}v_{r+2} z^{-j+1-2l}\\
&=\be_r(v_{-r-1}z^{-j}
+(q_0^{-1}-q_1)\sum_{l=1}^{-j}v_{r+1} z^{-j-2l}+(1-q_0^{-1}q_1)\sum_{l=2}^{-j}v_{r+1} z^{-j+1-2l})\\
&=\be_r(v_{r+1}z^j T_0).
}

The proposition is proved.
\endproof

By Proposition~\ref{commute} and the above identification $\bSjj \simeq \End_{\bbH}(\mathbb{V}^{\otimes d})$, there exists an $\cR$-algebra homomorphism
$$\Psi : \mathbb{U}^{\mathfrak{c}}(\mathfrak{\widehat{sl}}_n) \longrightarrow \bSjj.$$
The next lemma follows by a standard Vandermonde determinant type argument.

\begin{lemma}\label{xi(h_a)}
For each $\lambda \in \Lambda_{n,d}$, the element $\phi^e_{\lambda,\lambda} \in \bSjj$ belongs to the subalgebra of $\bSjj$ generated by $\Psi(\bh_a^{\pm 1})$, $0 \leq a\leq r+1$.
\end{lemma}


Here, we define two families of maps $\etil_i,\ \ftil_i : \Lambda_{n,d} \rightarrow \Lambda_{n,d} \sqcup \{ 0 \}$ ($0$ is a formal symbol) by
\begin{align}
\begin{split}
\etil_i(\lambda) &:= \begin{cases}
(\lambda_0, \ldots, \lambda_{i-1}, \lambda_i+1, \lambda_{i+1}-1,\lambda_{i+2}, \ldots, \lambda_{r+1}) &\tif \lambda_{i+1} > 0; \\
0 &\tif \lambda_{i+1} = 0,
\end{cases} \\
\ftil_i(\lambda) &:= \begin{cases}
(\lambda_0, \ldots, \lambda_{i-1}, \lambda_{i}-1, \lambda_{i+1}+1, \lambda_{i+2}, \ldots, \lambda_{r+1}) &\tif \lambda_i > 0; \\
0 &\tif \lambda_i = 0.
\end{cases}
\end{split} \nonumber
\end{align}
By convention, it is understood that $M_0 = 0$ and $\phi^g_{0,\mu} = 0 = \phi^g_{\lambda,0}$.

Recall the comultiplication $\Delta$ of $\mathbb{U}(\mathfrak{\widehat{sl}}_n)$ from Section~\ref{Quantum}. Then, we have
\begin{align}
\Delta^{(d-1)}(\bE_i) = \sum_{k = 0}^{d-1} 1^{\otimes k} \otimes \bE_i \otimes (\bK_i^{-1})^{\otimes d - k -1}, \quad \Delta^{(d-1)}(\bF_i) = \sum_{k = 0}^{d-1} \bK_i^{\otimes k} \otimes \bF_i \otimes 1^{\otimes d - k -1}. \nonumber
\end{align}

\begin{lemma}\label{xi(e_i)}
For $0 \leq i \leq r$, we have
\begin{align}
\begin{split}
\Psi(\be_i) &= \bc{
\sum_{\lambda \in \Lambda} q^{\lambda_{i+1}-1} \phi^e_{\etil_i(\lambda),\lambda} &\tif i \neq r;
\\
\sum_{\lambda \in \Lambda} q^{3(\lambda_{r+1}-1)}q_0q_1^{-1} \phi^e_{\etil_{r}(\lambda),\lambda} &\tif i = r,
} \\
\Psi(\bbf_i) &= \bc{
\sum_{\lambda \in \Lambda} q^{\lambda_{i}-1} \phi^e_{\ftil_i(\lambda),\lambda} &\tif i \neq 0,r;
\\
\sum_{\lambda \in \Lambda} q_1q^{2(\lambda_0 - 1)} \phi^e_{\ftil_0(\lambda),\lambda} &\tif i = 0;
\\
\sum_{\lambda \in \Lambda} q_0q_1^{-1}q^{\lambda_r - \lambda_{r+1} - 1} \phi^e_{\ftil_r(\lambda),\lambda} & \tif i = r.
} \nonumber
\end{split}
\end{align}
\end{lemma}

\begin{proof}
The proof is by a direct computation. Below we present the details only for verifying the most complicated equation
\[
\Psi(\be_r) = \sum_{\lambda \in \Lambda_{n,d}} q^{3(\lambda_{r+1}-1)}q_0q_1^{-1} \phi^e_{\etil_{r}(\lambda),\lambda}.
\]

First, we compute $\Psi(\be_r)$. It suffices to compute $\Psi(\be_r)(M_{\lambda})$ for all $\lambda \in \Lambda_{n,d}$. Since $\be_r = \bE_r + q^{-1}\bF_{-r-1}\bK_r^{-1}$, we have
\begin{align}
\begin{split}
\Psi(\be_r)(M_{\lambda}) &= \Delta^{(d-1)}(\bE_r)M_\lambda + q^{-1-\lambda_r + \lambda_{r+1}}\Delta^{(d-1)}(\bF_{-r-1})M_\lambda \\
&= \sum_{k=1}^{\lambda_{r+1}} q^{\lambda_{r+1}-k} M_{(0^{\lambda_0},\ldots, r^{\lambda_r}, r+1^{k-1}, r, r+1^{\lambda_{r+1}-k})}\\
&+ q^{-1-\lambda_r + \lambda_{r+1}} \sum_{k=1}^{\lambda_{r+1}} q^{k-1} M_{(0^{\lambda_0}, \ldots, r^{\lambda_r}, r+1^{k-1}, r+2, r+1^{\lambda_{r+1}-k})}.
\end{split} \nonumber
\end{align}

Next, we calculate $\phi^e_{\etil_r(\lambda),\lambda}$. It suffices to compute $\phi^e_{\etil_r(\lambda),\lambda}(M_\lambda)$. By the definition of $\phi^e_{\etil_r(\lambda),\lambda}$, it follows that
\begin{align}
\begin{split}
\phi^e_{\etil_r(\lambda),\lambda}(M_\lambda) &= \sum_{w \in W_\lambda \cap \D_{\etil_r(\lambda)}} q_w^{-1} M_{\etil_r(\lambda)} T_w.
\end{split} \nonumber
\end{align}
Note that
$W_\lambda \cap \D_{\etil_r(\lambda)} =
\{ s_{\lambda_{0,r} + 1} \cdots s_{\lambda_{0,r} + k-1}\}_{k= 1}^{\lambda_{r+1}}
\sqcup \{ s_{\lambda_{0,r} + 1} \cdots s_{d-1} s_d s_{d-1} \cdots s_{\lambda_{0,r} + k} \}_{k= 1}^{\lambda_{r+1}}
$.
Moreover, $s_{\lambda_{0,r} + 1} \cdots s_{\lambda_{0,r} + k-1}$ and $s_{\lambda_{0,r} + 1} \cdots s_{d-1} s_d s_{d-1} \cdots s_{\lambda_{0,r} + k}$ are reduced expressions for $1\leq k \leq \lambda_{r+1}$.
Hence, we have
\begin{align}
\begin{split}
\sum_{w \in W_\lambda \cap \D_{\etil_r(\lambda)}} q_w^{-1} M_{\etil_r(\lambda)} T_w &= \sum_{k=1}^{\lambda_{r+1}} q^{-k+1} M_{\etil_r(\lambda)} T_{\lambda_{0,r}+1} \cdots T_{\lambda_{0,r}+k-1} \\
&+ \sum_{k=1}^{\lambda_{r+1}} q^{-(2\lambda_{r+1}-k-1)}q_0 M_{\etil_r(\lambda)} T_{\lambda_{0,r}+1} \cdots T_{d-1}T_dT_{d-1} \cdots T_{\lambda_{0,r}+k}.
\end{split} \nonumber
\end{align}
It is easily verified that $M_{\etil_r(\lambda)} T_{\lambda_{0,r}+1} \cdots T_{\lambda_{0,r}+k-1} = M_{(0^{\lambda_0}, \ldots, r^{\lambda_r}, r+1^{k-1}, r, r+1^{\lambda_{r+1}-k})}$. In order to compute the other terms, we first note that, from \eqref{Td2},
\begin{align*}
T_{\lambda_{0,r}+1} &\cdots T_{d-1}T_dT_{d-1} \cdots T_{\lambda_{0,r}+k}
\\
&= q_0^{-1} T_{\lambda_{0,r}+1} \cdots T_{d-1}  (T_{d-1} \cdots T_1 X_1T_0^{-1} T_1^{-1} \cdots T_{d-1}^{-1})  T_{d-1} \cdots T_{\lambda_{0,r}+k}
\\
&= q_0^{-1} T_{\lambda_{0,r}+1} \cdots T_{d-1}  T_{d-1} \cdots T_1 X_1T_0^{-1} T_1^{-1} \cdots T_{\lambda_{0,r}+k-1}^{-1}.
\end{align*}
It then follows from \eqref{T0X1} that
\begin{align*}
&T_{\lambda_{0,r}+1} \cdots T_{d-1}T_dT_{d-1} \cdots T_{\lambda_{0,r}+k}
\\
&= q_0^{-1} T_{\lambda_{0,r}+1} \cdots T_{d-1}  T_{d-1} \cdots T_1
q_0q_1^{-1}(T_0X_1^{-1} -(q_0^{-1}q_1 - 1))
T_1^{-1} \cdots T_{\lambda_{0,r}+k-1}^{-1} \\
&= q_1^{-1} T_{\lambda_{0,r}+1} \cdots T_{d-1}  T_{d-1} \cdots T_1 T_0X_1^{-1} T_1^{-1} \cdots T_{\lambda_{0,r}+k-1}^{-1}
\\
&+ (q_1^{-1} - q_0^{-1}) T_{\lambda_{0,r}+1} \cdots T_{d-1}  T_{d-1} \cdots T_{\lambda_{0,r}+k}.
\end{align*}
Since for $a \leq b$ we have
$
T_a \cdots T_b T_b \cdots T_a = 1 + (q\inv - q) \sum_{l = a}^b T_a \cdots T_l \cdots T_a,
$
we obtain
\begin{align*}
&T_{\lambda_{0,r}+1} \cdots T_{d-1}T_dT_{d-1} \cdots T_{\lambda_{0,r}+k}
\\
&= q_1^{-1} \Bp{
1 + (q^{-1} - q)\sum_{l=1}^{{\lambda_{r+1}-1}}
T_{\lambda_{0,r}+1} \cdots T_{\lambda_{0,r}+l} \cdots T_{\lambda_{0,r}+1}
}
T_{\lambda_{0,r}} \cdots T_1 T_0X_1^{-1} T_1^{-1} \cdots T_{\lambda_{0,r}+k-1}^{-1}
\\
&+ (q_1^{-1} - q_0^{-1})
T_{\lambda_{0,r}+1} \cdots T_{\lambda_{0,r}+k-1}
\Bp{
 1 + (q^{-1} - q)
 \sum_{l=k}^{\lambda_{r+1}-1}
 T_{\lambda_{0,r}+k} \cdots  T_{\lambda_{0,r}+l}  \cdots T_{\lambda_{0,r}+k}
 }
\\
&= q_1^{-1} T_{\lambda_{0,r}} \cdots T_1 T_0X_1^{-1} T_1^{-1} \cdots T_{\lambda_{0,r}+k-1}^{-1}
\\
&+q_1^{-1}(q^{-1}-q) \sum_{l=1}^{{\lambda_{r+1}-1}} T_{\lambda_{0,r}+1} \cdots T_{\lambda_{0,r}+l-1} T_{\lambda_{0,r}+l} T_{\lambda_{0,r}+l-1} \cdots T_1 T_0X_1^{-1} T_1^{-1} \cdots T_{\lambda_{0,r}+k-1}^{-1}
\\
&+(q_1^{-1}-q_0^{-1})T_{\lambda_{0,r}+1} \cdots T_{\lambda_{0,r}+k-1}
\\
& +(q_1^{-1}-q_0^{-1})(q^{-1}-q)\sum_{l=k}^{\lambda_{r+1}-1}T_{\lambda_{0,r}+1} \cdots T_{\lambda_{0,r}+l-1} T_{\lambda_{0,r}+l} T_{\lambda_{0,r}+l-1} \cdots T_{\lambda_{0,r}+k}.
\end{align*}
Therefore, we need to compute
\begin{align}
\label{1st term}
&M_{\etil_r(\lambda)}T_{\lambda_{0,r}} \cdots T_1 T_0X_1^{-1} T_1^{-1} \cdots T_{\lambda_{0,r}+k-1}^{-1},
\\
\label{extra terms}
&M_{\etil_r(\lambda)} T_{\lambda_{0,r}+1} \cdots T_{\lambda_{0,r}+l-1} T_{\lambda_{0,r}+l} T_{\lambda_{0,r}+l-1} \cdots T_{\lambda_{0,r}+k},
\\
\label{other terms}
&M_{\etil_r(\lambda)} T_{\lambda_{0,r}+1} \cdots T_{\lambda_{0,r}+l-1} T_{\lambda_{0,r}+l} T_{\lambda_{0,r}+l-1} \cdots T_1 T_0X_1^{-1} T_1^{-1} \cdots T_{\lambda_{0,r}+k-1}^{-1},
\end{align}
which are given as below:
\begin{align}
\begin{split}
\eqref{1st term} &= M_{\etil_r(\lambda)}T_{\lambda_{0,r}} \cdots T_1 T_0 T_1 \cdots T_{\lambda_{0,r}+k-1} X_{\lambda_{0,r}+k}^{-1}
= q^{-\lambda_r} M_{0^{\lambda_0},\ldots,r^{\lambda_r},r+1^{k-1},-r+n,r+1^{\lambda_{r+1}-k}},
\\
\eqref{extra terms} &= M_{(0^{\lambda_0},\ldots,r^{\lambda_r},r+1^l,r,r+1^{\lambda_{r+1}-l-1})} T_{\lambda_{0,r}+l-1} \cdots T_{\lambda_{0,r}+k}
= q^{-(l-k)} M_{(0^{\lambda_0},\ldots,r^{\lambda_r},r+1^l,r,r+1^{\lambda_{r+1}-l-1})},
\\
\eqref{other terms} &= M_{(0^{\lambda_0},\ldots,r^{\lambda_r},r+1^l,r,r+1^{\lambda_{r+1}-l-1})} T_{\lambda_{0,r}+l-1} \cdots T_0 X_1^{-1} T_1^{-1} \cdots T_{\lambda_{0,r}+k-1} \\
&= q^{-(l-1)} M_{(r+1,0^{\lambda_0},\ldots,r^{\lambda_r},r+1^{l-1},r,r+1^{\lambda_{r+1}-l-1})} T_0X_1^{-1} T_1^{-1} \cdots T_{\lambda_{0,r}+k-1}
\\
&= q^{-(l-1)}q_0^{-1}q_1 M_{(r+1-n,0^{\lambda_0},\ldots,r^{\lambda_r},r+1^{l-1},r,r+1^{\lambda_{r+1}-l-1})} X_1^{-1} T_1^{-1} \cdots T_{\lambda_{0,r}+k-1}
\\
&= \begin{cases}
q^{k-l-1}q_0^{-1}q_1 M_{(0^{\lambda_0}, \ldots, r^{\lambda_r}, r+1^{l-1}, r, r+1^{\lambda_{r+1}-l})} &\tif 1\leq l \leq k-1;
\\
q^{k-l}q_0^{-1}q_1 M_{(0^{\lambda_0}, \ldots, r^{\lambda_r}, r+1^{l}, r, r+1^{\lambda_{r+1}-l-1})} &\tif k\leq l \leq \lambda_{r+1}-1.
\end{cases}
\end{split} \nonumber
\end{align}
\allowdisplaybreaks
Summarizing the calculations above, we have
\begin{align*}
\phi^e_{\etil_r(\lambda),\lambda}(M_\lambda)
&= \sum_{k = 1}^{\lambda_{r+1}}
q^{-k+1}
M_{(0^{\lambda_0}, \ldots, r^{\lambda_r}, r+1^{k-1}, r, r+1^{\lambda_{r+1}-k})}
\\
&+ \sum_{k=1}^{\lambda_{r+1}}
q^{-2\lambda_{r+1}-\lambda_r+k+1}q_0q_1^{-1}
M_{(0^{\lambda_0}, \ldots, r^{\lambda_r}, r+1^{k-1}, -r+n, r+1^{\lambda_{r+1}-k})}
\\
&+\sum_{k=1}^{\lambda_{r+1}}
q^{-2\lambda_{r+1}+k+1}(q^{-1}-q) \cdot \\
&\Bp{
\sum_{l = 1}^{k-1}
q^{k-l-1}
M_{(0^{\lambda_0}, \ldots, r^{\lambda_r}, r+1^{l-1}, r, r+1^{\lambda_{r+1}-l})}
+ \sum_{l = k}^{\lambda_{r+1}-1}
q^{k-l}
M_{(0^{\lambda_0}, \ldots, r^{\lambda_r}, r+1^l, r, r+1^{\lambda_{r+1}-l-1})} }
\\
&+ \sum_{k=1}^{\lambda_{r+1}} q^{-2\lambda_{r+1}+k+1}(q_0q_1^{-1}-1) M_{(0^{\lambda_0},\ldots,r^{\lambda_r},r+1^{k-1},r,r+1^{\lambda_{r+1}-k})}
\\
&+ \sum_{k=1}^{\lambda_{r+1}} q^{-2\lambda_{r+1}+k+1}(q_0q_1^{-1}-1)(q^{-1}-q) \sum_{l=k}^{\lambda_{r+1}-1} q^{-l+k} M_{(0^{\lambda_0},\ldots,r^{\lambda_r},r+1^l,r,r+1^{\lambda_{r+1}-l-1})}.
\end{align*}
The coefficient for $M_{(0^{\lambda_0}, \ldots, r^{\lambda_r}, r+1^{k-1}, -r+n, r+1^{\lambda_{r+1}-k})}$ is
$q^{-3(\lambda_{r+1}-1)}q_0q_1^{-1} q^{-1-\lambda_r+\lambda_{r+1}+k-1}$
while that for $M_{(0^{\lambda_0}, \ldots, r^{\lambda_r}, r+1^{k-1}, r, r+1^{\lambda_{r+1}-k})}$ is
\begin{align}
\begin{split}
&q^{-k+1} + q^{-2\lambda_{r+1}+k+1}(q_0q_1^{-1}-1) + \sum_{m=k+1}^{\lambda_{r+1}} q^{-2\lambda_{r+1}-k+2m}(q^{-1}-q)\\ &+\sum_{m=1}^{k-1} q^{-2\lambda_{r+1}-k+2+2m}(q^{-1}-q) + \sum_{m=1}^{k-1} q^{-2\lambda_{r+1}-k+2+2m}(q^{-1}-q)(q_0q_1^{-1}-1).
\end{split} \nonumber
\end{align}
One checks that the latter coincides with $q^{-2\lambda_{r+1}-k+3}q_0q_1^{-1} = q^{-3(\lambda_{r+1}-1)}q_0q_1^{-1}  q^{\lambda_{r+1}-k}$.
Thus, we obtain
\begin{align}
\begin{split}
\phi^e_{\etil_r(\lambda),\lambda}(M_\lambda) = q^{-3(\lambda_{r+1}-1)}q_0q_1^{-1} \Psi(\be_r)(M_\lambda),
\end{split} \nonumber
\end{align}
which proves the assertion.
\end{proof}

\begin{prop}\label{generator of bSjj}
Assume $r\geq d$.
Then the Schur algebra $\bSjj$ is generated by $\Psi(\be_i)$, $\Psi(\bbf_i)$, $\Psi(\bh_a^{\pm 1})$ for $0 \leq i \leq r$ and $0 \leq a \leq r+1$.
\end{prop}

\begin{proof}
Let $S$ denote the subalgebra of $\bSjj$ generated by $\Psi(\be_i), \Psi(\bbf_i), \Psi(\bh_a^{\pm 1})$,  for $0 \leq i \leq r$, and for $0 \leq a \leq r+1$. By Lemma \ref{xi(h_a)} and \ref{xi(e_i)}, for each $\lambda \in \Lambda_{n,d}$ and $0\leq i \leq r$, we have $\phi^e_{\etil_i(\lambda),\lambda} \in S$ and $\phi^e_{\ftil_i(\lambda),\lambda} \in S$.

Take $\lambda \in \Lambda_{n,d}$ arbitrarily. It is easy to check that there exists a sequence $(x_1,\ldots,x_l)$ of $\etil_i, \ftil_i$'s such that $\lambda = x_1 \cdots x_l(\omega)$ and $\{ e \} = W_{\lambda_0} \subset \cdots \subset W_{\lambda_l} = W_{\lambda}$, where $\lambda_0 = \omega$ and $\lambda_{i} = x_i(\lambda_{i-1})$. Then, we have
$$\phi^e_{\lambda,\omega} = \phi^e_{\lambda_r,\lambda_{r-1}} \cdots \phi^e_{\lambda_1,\lambda_0} \in S.$$
By the same way, we obtain $\phi^e_{\omega,\lambda} \in S$.

Next, for $0\leq i \leq {d}-1$, we have
\[
\phi^{e}_{\omega,\etil_i(\omega)} \cdot \phi^e_{\etil_i(\omega),\omega}
= \phi^{{e}}_{\omega,\omega} + q_{s_i}^{-1}\phi^{s_i}_{\omega,\omega}.
\]
We also have
\[
\phi^e_{\omega,\ftil_r(\omega)} \cdot \phi^e_{\ftil_r(\omega),\omega}
= \phi^{{e}}_{\omega,\omega} + \phi^{s_d}_{\omega,\omega}.
\]
These show that $\phi^{s_i}_{\omega,\omega} \in S$ for $0 \leq i \leq {d}$. Since $\phi^e_{\omega,\omega} \bSjj \phi^e_{\omega,\omega}$ is generated by $\phi^{s_i}_{\omega,\omega} \in S$, $0 \leq i \leq r$, we have $\phi^e_{\omega,\omega} \bSjj \phi^e_{\omega,\omega} \subset S$.

Finally, for each $\lambda, \mu \in \Lambda_{n,d}$, we have $\phi^e_{\lambda,\omega} \bSjj \phi^e_{\omega,\mu} = \phi^e_{\lambda, \omega} \phi^e_{\omega,\omega} \bSjj \phi^e_{\omega,\omega} \phi^e_{\omega,\mu} \subset S$. Since $\bSjj$ is the direct sum of $\phi^e_{\lambda,\omega} \bSjj \phi^e_{\omega,\mu}$, we conclude that $S = \bSjj$.
\end{proof}

%
%

\begin{theorem}
 \label{thm:3-par}
Suppose $r \geq d\geq 1$.
We have the following Schur $(\UUC, \bbH)$-duality:
\begin{align*}
\Psi (\UUC) \simeq & \End_{\bbH} (\mathbb{V}^{\otimes d}),
\\
& \End_{\UUC} (\mathbb{V}^{\otimes d})  \simeq \bbH^{op}.
\end{align*}
\end{theorem}
\noindent (To be consistent with the variants in next section, we can refer to this as $\jjw$-Schur duality.)

\begin{proof}
It follows by Proposition \ref{generator of bSjj} that $\Psi (\UUC) =\bSjj$. Hence the first isomorphism follows by Lemma~\ref{lem:iso}.

Since $\Psi (\UUC) =\bSjj$ induces an isomorphism $\End_{\mathbb{U}^{\mathfrak{c}}(\mathfrak{\widehat{sl}}_n)}(\mathbb{V}^{\otimes d}) \simeq \End_{\kappa\bSjj\kappa^{-1}}(\mathbb{V}^{\otimes d})$, we have
\[
\End_{\mathbb{U}^{\mathfrak{c}}(\mathfrak{\widehat{sl}}_n)}(\mathbb{V}^{\otimes d})
\simeq \End_{\kappa\bSjj\kappa^{-1}}(\mathbb{V}^{\otimes d})
\simeq \End_{\bSjj}(\bTjj)
\simeq (\phi^e_{\omega,\omega} \bSjj \phi^e_{\omega,\omega})^{op} \simeq \bbH^{op}.
\]
The second isomorphism follows.
\end{proof}

\subsection{Specializations}
 \label{sec:special}

When specializing $\bbH$ to the single parameter case by letting $q_0=1$ and $q_1=q^2$, we obtain the affine Hecke algebra of type C over $\QQ(q)$, denoted here by $\bH_{C_d}$. This is the Hecke algebra appearing in \cite{FL3Wa}-\cite{FL3Wb}.

When specializing $\bbH$ to $q_0=q_1$, we obtain the extended affine Hecke algebra of type B over $\QQ(q,q_1)$ in 2 parameters $q, q_1$.
When specializing $\bbH$ to the single parameter case by letting $q_0=q_1=q$, we obtain the extended affine Hecke algebra of type B over $\QQ(q)$, denoted here by $\bH_{B_d}$.

When specializing $\bbH$ to $q_0=q_1 =1$, we obtain the extended affine Hecke algebra of type D over $\QQ(q)$, denoted here by $\bH_{D_d}$.

Specializing our main Theorem~\ref{thm:3-par} on the 3-parameter Schur duality to 2-parameter or 1-parameter cases, we obtain several versions of dualities, each of which is meaningful in its own way. In this sense, the duality in Theorem~\ref{thm:3-par} is a master duality which unifies dualities of different types (among which the 1-parameter dualities should admit geometric interpretations using different types of flags).

The framework in \cite{FL3Wa} provides a geometric setting for the $(\UUC|_{q_0=1,q_1=q^2}, \bH_{C_d})$-duality on $\mathbb{V}|_{q_0=1,q_1=q^2}^{\otimes d}$.  Both $\UUC|_{q_0=1, q_1=q^2}$ and $\bH_{C_d}$ are geometrically realized; while not discussed explicitly therein, $\mathbb{V}|_{q_0=1,q_1=q^2}^{\otimes d}$ can also be geometrically realized in terms of varieties of pairs of an ``$n$-step" partial flag and a complete flag.

\begin{rem}
Our work can lead to several interesting future projects, which are highly nontrivial to carry out. One bonus of carrying out these geometric constructions will be the positivity of the resulting $\imath$canonical bases.

\begin{enumerate}
\item
A geometric setting in flag variety of affine type B similar to \cite{FL3Wa}  for the $(\UUC|_{q_0=q_1=q}, \bH_{B_d})$-duality on $\mathbb{V}|_{q_0=q_1=q}^{\otimes d}$ is expected.
\item
A geometric setting in flag variety of affine type D similar to \cite{FL3Wa}  for the $(\UUC|_{q_0=q_1=1}, \bH_{D_d})$-duality on $\mathbb{V}|_{q_0=q_1=1}^{\otimes d}$ is expected. The finite type version of this duality would be a modification of the construction in \cite{FL15}.
\item
The algebraic construction in \cite{FL3Wb} is expected to generalize to the 3-parameter case or various 2-parameter or equal parameter specializations.
\item
Classify the finite-dimensional irreducible $\UUC$-modules.
\item
All remarks in \S\ref{sec:special} here are valid for the variants of Schur $(\UUC, \bbH)$-duality considered in Section~\ref{sec:var} below.
\end{enumerate}
\end{rem}

\section{Variants of Schur dualities}\label{sec:var}

Motivated by \cite{FL3Wa}--\cite{FL3Wb}, we formulate in this section several variants of the Schur $(\UUC, \bbH)$-duality in Theorem~\ref{thm:3-par}. We continue to assume $r\ge d \ge 1$. Furthermore we set
\[
\nn =n-1 =2r+1, \qquad \eta =n-2 =2r.
\]

\subsection{The $\jiw$-Schur duality}
 \label{sec:ji}

Let ${}'\mathbb{V}_{\nn}$ be the $\cR$-subspace of $\mathbb{V}$ spanned by $v_i$, for $i\in \ZZ$ such that  $i \not\equiv r+1 \pmod n$.
Note that  ${}'\mathbb{V}_{\nn}$ is naturally an $\bbH$-submodule of $\mathbb{V}$, and moreover, it is a direct sum of permutation modules.

We consider an isomorphic copy of $\UUAijg$ (with a different indexing set for generators), denoted by $\UUAjig$. The algebra $\UUAjig$ is generated by $\bE_i, \bF_i$
$(i \in [0,n-1]\backslash \{r+1\}),  \bD_a^{\pm1}$ $(a \in [0,{n-1}]\backslash \{r+1\});$
here we regard indices $r, r+2$ adjacent. Denote by $\UUAji$ the subalgebra of $\UUAjig$  {generated by $\bE_i, \bF_i, \bK_i$ $(i \in [0,n-1]\backslash \{r+1\})$, where $\bK_{r}=\bD_{r}\bD_{r+2}^{-1}, \bK_i=\bD_{i}\bD_{i+1}^{-1}$ ($i\neq r$).}
Then ${}'\mathbb{V}_{\nn}$ is a natural representation of $\UUAjig$, with the action given by: for $i \in [0,n-1]\backslash \{r+1\},  a \in [0,n]\backslash \{r+1\}$,
\eq\label{actionUji}
\bE_{i}v_{j+1}=
\bc{v_{j}&\tif j\equiv i \neq r;
\\
v_{j-1}&\tif j-1\equiv i = r;
\\
0&\text{else},}
\;
\bF_{i}v_{j}=\bc{
v_{j+1} &\tif j\equiv i \neq r;
\\
v_{j+2} &\tif j\equiv i =r;
\\
0 &\text{else},
}
\;
\bD_{a}v_{j}=
\bc{q v_{j} &\tif j\equiv a;
\\
v_{j} &\text{else}.
}
\endeq
Then $\UUAjig$ and $\UUAji$ act on ${}'\mathbb{V}_{\nn}^{\otimes d}$ via iterated comultiplication.

For $i, j \in [0, r]$, we denote the Cartan integers by
\begin{equation}
  \label{aij}
  \texttt{c}_{ij} = 2 \delta_{ij} - \delta_{i, j+1} - \delta_{i, j-1}.
\end{equation}

Define $\UUji$ (cf. \cite[Chapter 7]{FL3Wa}) to be the $\cR$-algebra generated by $\eji_i, \fji_i$, and $\kji^{\pm 1}_i$ $(0 \le i\le r-1)$ and $\tji_r$, subject to the following relations: for all $0 \le i,j \le r-1$,
\begin{align*}
& \kji_0  (\kji^2_1 \cdots \kji^2_{r-1} )   = q^{-1}, 
\quad
\kji_i \kji_i^{-1}  = 1, \quad
\kji_i \kji_j   = \kji_j \kji_i, \quad
\kji_i \tji_r = \tji_r \kji_i,  \\
&\kji_i \eji_j \kji_i^{-1}=  q^{\texttt{c}_{ij} + \delta_{i,0} \delta_{j,0}}  \eji_j ,
\quad
\kji_i \fji_j \kji_i^{-1}  =  q^{-\texttt{c}_{ij} - \delta_{i,0} \delta_{j,0} }  \fji_j , \\
&\eji_i \eji_j  = \eji_j \eji_i , \quad
\fji_i \fji_j  = \fji_j \fji_i , \quad \forall |i-j|> 1, \\
&\eji_i \fji_j - \fji_j \eji_i  = \delta_{ij} \frac{\kji_i - \kji_i^{-1}}{q - q^{-1}} , \quad \forall (i, j) \neq (0,0),
\\
&\eji_i \tji_r  = \tji_r \eji_i, \quad
\fji_i \tji_r  = \tji_r \fji_i, \quad \forall i \leq r-2,
\\
&\eji^2_{r-1} \tji_r + \tji_r \eji^2_{r-1}  = (q + q^{-1}) \eji_{r-1} \tji_r \eji_{r-1},
\quad
\fji^2_{r-1} \tji_r + \tji_r \fji^2_{r-1}  = (q + q^{-1}) \fji_{r-1} \tji_r \fji_{r-1},
\\
&\tji_r^2 \eji_{r-1} + \eji_{r-1} \tji_r^2  = (q + q^{-1}) \tji_r \eji_{r-1} \tji_r + q_0 q_1^{-1}  \eji_{r-1},
\\
&\tji_r^2 \fji_{r-1} + \fji_{r-1} \tji_r^2  = (q + q^{-1}) \tji_r \fji_{r-1} \tji_r + q_0 q_1^{-1} \fji_{r-1},
%
\\
&\eji_i^2 \eji_j  + \eji_j \eji_i^2 = (q + q^{-1}) \eji_i \eji_j \eji_i,
\quad
\fji_i^2 \fji_j  + \fji_j \fji_i^2 = (q + q^{-1}) \fji_i \fji_j \fji_i, \quad \forall |i-j|=1,
\\
&\eji_0^2 \fji_0 + \fji_0 \eji_0^2
= (q + q^{-1}) \big( \eji_0 \fji_0 \eji_0 - q_1 q \eji_0 \kji_0 -q_0^{-1}  q^{-1} \eji_0 \kji_0^{-1} \big),
\\
&\fji_0^2 \eji_0   + \eji_0 \fji_0^2
 = (q + q^{-1}) \big ( \fji_0 \eji_0 \fji_0 - q q_1 \kji_0 \fji_0 - q_0^{-1} q^{-1} \kji_0^{-1} \fji_0 \big).\\
\end{align*}

\begin{prop}
 \label{prop:homji}
There is an injective $\cR$-algebra homomorphism
$\jiw: \UUji \rightarrow \UUAji$ such that
\bA{
\notag
&\bk_a \mapsto \bK_a\bK_{-a-1}^{-1},   \quad (0\leq a \le  r-1)
\\
\notag
&\be_i\mapsto \bE_i+\bF_{-i-1}\bK_i^{-1},
\quad
\bbf_i\mapsto \bE_{-i-1} + \bF_i\bK_{-i-1}^{-1},
\quad (1\leq i \leq r-1)
\\
\notag
&\be_0\mapsto\bE_0+q_0^{-1}\bF_{-1}\bK_{0}^{-1},
\quad
\bbf_0\mapsto \bE_{-1} + q_1q^{-1}\bF_0\bK_{-1}^{-1},
\\
& \label{embeding4ji}
\bt_r \mapsto \bE_r+qq_0q_1^{-1} \bF_r\bK_{r}^{-1}+(1-q_0q_1^{-1})/(q-q^{-1})\bK_{r}^{-1}.
}
\end{prop}

\begin{proof}
The proof is similar to the proof for Proposition \ref{prop:homjj}. The subalgebra here is a quantum symmetric pair coideal subalgebra associated with the Dynkin diagram and involution below, and the proposition follows from {\cite[Theorem~ 7.8]{Ko14}}.
\begin{figure}[ht!]
\caption{Dynkin diagram of type $A^{(1)}_{2r}$ with involution of type $\jmath\imath$.}
   \label{figure:ji}
\begin{tikzpicture}
\matrix [column sep={0.6cm}, row sep={0.5 cm,between origins}, nodes={draw = none,  inner sep = 3pt}]
{
	\node(U1) [draw, circle, fill=white, scale=0.6, label = 0] {};
	&\node(U2)[draw, circle, fill=white, scale=0.6, label =1] {};
	&\node(U3) {$\cdots$};
	&\node(U5)[draw, circle, fill=white, scale=0.6, label =$r-1$] {};
\\
	&&&&
	\node(R)[draw, circle, fill=white, scale=0.6, label =$r$] {};
\\
	\node(L1) [draw, circle, fill=white, scale=0.6, label =below:$2r$] {};
	&\node(L2)[draw, circle, fill=white, scale=0.6, label =below:$2r-1$] {};
	&\node(L3) {$\cdots$};
	&\node(L5)[draw, circle, fill=white, scale=0.6, label =below:$r+1$] {};
\\
};
\begin{scope}
\draw (U1) -- node  {} (U2);
\draw (U2) -- node  {} (U3);
\draw (U3) -- node  {} (U5);
\draw (U5) -- node  {} (R);
\draw (U1) -- node  {} (L1);
\draw (L1) -- node  {} (L2);
\draw (L2) -- node  {} (L3);
\draw (L3) -- node  {} (L5);
\draw (L5) -- node  {} (R);
\draw (R) edge [color = blue,loop right, looseness=40, <->, shorten >=4pt, shorten <=4pt] node {} (R);
\draw (L1) edge [color = blue,<->, bend right, shorten >=4pt, shorten <=4pt] node  {} (U1);
\draw (L2) edge [color = blue,<->, bend right, shorten >=4pt, shorten <=4pt] node  {} (U2);
\draw (L5) edge [color = blue,<->, bend left, shorten >=4pt, shorten <=4pt] node  {} (U5);
\end{scope}
\end{tikzpicture}
\end{figure}
\end{proof}

Recalling $\Ld_{n,d}$ from \eqref{def:Ld},
we define
\begin{align*}
\Lambda^{\jmath\imath}_{\nn,d} &= \{ \lambda =(\lambda_0, \lambda_1, \ldots, \lambda_{r+1}) \in \Ld_{n,d}~|~\lambda_{r+1}=0\}.
\end{align*}
Note that $\omega \in \Lambda^{\jmath\imath}_{\nn,d}$.
Define the right $\bbH$-module
\begin{align*}
\bTji = \bigoplus_{\lambda \in \Lambda^{\jmath\imath}_{\nn,d}}  x_{\lambda}\bbH.
\end{align*}
Following \cite{FL3Wb}, we define the $\jiw$-variant of the Schur algebra $\bSjj$ as follows:
\begin{align*}
\bSji &= \End_{\bbH}(\oplus_{\lambda \in \Lambda^{\jmath\imath}_{\nn,d}} x_{\lambda}\bbH) =\bigoplus_{\lambda,\mu \in \Lambda^{\jmath\imath}_{\nn,d}} \Hom_{\bbH}(x_{\mu}\bbH,x_{\lambda}\bbH).
\end{align*}
It is routine to show that $\{\phi^g_{\lambda,\mu} \mid \lambda,\mu \in \Lambda^{\ji}_{\nn,d}, g \in \D_{\lambda,\mu} \}$ form an $\cR$-basis of $\bSji$.

The following is a variant of Lemma~\ref{lem:iso}.
\begin{lem}
\label{lem:kappaji}
We have an isomorphism of $\bbH$-modules:
$\bTji \cong {}'\mathbb{V}_{\nn}^{\otimes d}.$
\end{lem}

Note $\UUji$ acts on ${}'\mathbb{V}_{\nn}^{\otimes d}$ via the embedding $\jiw: \UUji \rightarrow \UUAji$; we denote this action by $\Psi^{\jmath\imath}$.

\begin{thm}
 \label{thm:3-par:ji}
We have the following Schur $(\UUji, \bbH)$-duality:
\begin{align*}
\Psi^{\jmath\imath}(\UUji) \simeq & \End_{\bbH} ({}'\mathbb{V}_{\nn}^{\otimes d}),
\\
& \End_{\UUji} ({}'\mathbb{V}_{\nn}^{\otimes d})  \simeq \bbH^{op}.
\end{align*}
\end{thm}

\begin{proof}
We first check that the actions of $\UUji$ and $\bbH$ on $'\mathbb{V}_{\nn}^{\otimes d}$ commute.
As seen in Proposition~\ref{commute}, it remains to verify that the $T_0$-action commutes with the $\bt_r$-action on $'\mathbb{V}_{\nn}$.
For the unique expression $f = k + jn$ such that $-r \leq k \leq r$, we combine \eqref{embeding4ji} and \eqref{actionUji} and then obtain
\begin{align}
&
\bt_r v_f = \frac{1-q_0q_1\inv}{q-q\inv} v_f
\quad
(f \not\equiv r,r+2),
\\
&
\bt_r v_rz^j
=
q_0q_1\inv v_{-r}z^{j+1}
+ \frac{1-q_0q_1\inv}{q-q\inv}q\inv v_rz^j,
\quad
\bt_r v_{-r}z^j
=
v_rz^{j-1} + \frac{1-q_0q_1\inv}{q-q\inv} q v_{-r}z^j.
\end{align}
Note that the $\bt_r$-action is a scalar multiplication for $f \not\equiv r,r+2$ and hence it commutes with the $T_0$-action \eqref{T0onV}. For $k = r, j \geq 0$ we have
\begin{align*}
(\bt_r v_rz^j) T_0
&=
q_0q_1\inv \big{(}
q_0^{-1}q_1v_{r}z^{-j-1}
+(q_1-q_0^{-1})\sum_{l=1}^{j}v_{-r}z^{j-2l+1}
+(q_0^{-1}q_1-1)\sum_{l=1}^{j+1}v_{-r}z^{j-2l+2}
\big{)}
\\
&+ \frac{1-q_0q_1\inv}{q-q\inv}q\inv \big{(}
v_{-r}z^{-j}
+(q_1-q_0^{-1})\sum_{l=1}^{j}v_rz^{j-2l}
+(q_0^{-1}q_1-1)\sum_{l=1}^{j}v_rz^{j-2l+1}
\big{)},
\\
\bt_r (v_rz^j T_0)
&=
v_rz^{-j-1} + \frac{1-q_0q_1\inv}{q-q\inv} q v_{-r}z^{-j}
\\
&+(q_1-q_0^{-1})\sum_{l=1}^{j}
\big{(}
q_0q_1\inv v_{-r}z^{j-2l+1}
+ \frac{1-q_0q_1\inv}{q-q\inv}q\inv v_rz^{j-2l}
\big{)}
\\
&+(q_0^{-1}q_1-1)\sum_{l=1}^{j}
\big{(}
q_0q_1\inv v_{-r}z^{j-2l+2}
+ \frac{1-q_0q_1\inv}{q-q\inv}q\inv v_rz^{j-2l+1}
\big{)}.
\end{align*}
They are indeed equal.
The rest can be checked similarly and the commutivity follows.

For the first isomorphism, it suffices to show that
\begin{equation}
 \label{eq:jiS}
\Psi^{\jmath\imath}(\UUji) \simeq \bSji,
\end{equation}
which follows from a variant of Proposition~\ref{generator of bSjj} as below.
Let $'\!S$ be the subalgebra of $\bSji$ generated by $\Psi^{\jmath\imath}(\be_i), \Psi^{\jmath\imath}(\bbf_i), \Psi^{\jmath\imath}(\bk_i)$ and $\Psi^{\jmath\imath}(\bt_r)$ for $0\leq i \leq r-1$. Similar to the proof of Proposition~\ref{generator of bSjj}, one can show that ${}'\!S$ contains the elements $\phi^e_{\omega, \ld}, \phi^e_{\ld, \omega}$ for all $\ld \in \Ld^{\ji}_{\nn,d}$ and the elements $\phi^{s_i}_{\omega, \omega}$ for $0 \leq i \leq d-1$.
The only difference here is that $\phi^{s_d}_{\omega, \omega} \in {}'\!S$ follows from
\begin{equation*}
\Psi^{\ji}(\bt_r) \in \sum_{\ld \in \Ld^{\ji}_{\nn, d}}\cR \phi^{s_d}_{\ld, \ld}.
\end{equation*}

For the second isomorphism, note that \eqref{eq:jiS} together with Lemma~ \ref{lem:kappaji} induce an isomorphism
$\End_{\UUji}('\VV_{\nn}^{\otimes d}) \simeq \End_{\bSji}(\bTji)$.
The theorem now follows since
\[
\End_{\bSji}(\bTji)
\simeq (\phi^e_{\omega,\omega} \bSji \phi^e_{\omega,\omega})^{op} \simeq \bbH^{op}.
\]
\end{proof}

\subsection{The $\ijw$-Schur duality}

Let $\mathbb{V}_{\nn}$ be the $\cR$-subspace of $\mathbb{V}$ spanned by $v_i$, for $i\in \ZZ$ such that  $i \not\equiv 0 \pmod n$.
Note that $\mathbb{V}_{\nn}$ is  naturally an $\bbH$-submodule of $\mathbb{V}$, and moreover, it is a direct sum of permutation modules.

Recall $\UUAijg$ is generated by $\bE_i, \bF_i (0 \leq i \leq \nn-1), \bD_a^{\pm1} (1 \leq a \leq \nn).$ Denote by $\UUAij$ the subalgebra of $\UUAijg$  {generated by $\bE_i, \bF_i, \bK_i$ $(i \in [0,\nn-1])$, where $\bK_{0}=\bD_{n-1}\bD_{1}^{-1}, \bK_i=\bD_{i}\bD_{i+1}^{-1}$ ($i\neq 0$).}
Then $\mathbb{V}_{\nn}$ is a natural representation of $\UUAijg$, with the action given by
\eq\label{actionUij}
\bE_{i}v_{j+1}=
\bc{v_{j}&\tif j\equiv i \neq 0;
\\
v_{j-1}&\tif j\equiv i = 0;
\\
0&\text{else},}
\quad
\bF_{i}v_{j}=\bc{
v_{j+1} &\tif j\equiv i \neq 0;
\\
v_{j+2} &\tif j+1\equiv i =0;
\\
0 &\text{else},
}
\quad
\bD_{a}v_{j}=
\bc{q v_{j} &\tif j\equiv a;
\\
v_{j} &\text{else}.
}
\endeq
Then $\UUAijg$ and $\UUAij$ act on $\mathbb{V}_{\nn}^{\otimes d}$ via iterated comultiplication.


Define $\UUij$ to be an $\cR$-algebra generated by $\bt_0,
\be_i, \bbf_i,$ and $\bk_i^{\pm 1} \;  (1 \leq i \leq r)$. 
We will not write down all its relations explicitly, as there is a $\QQ(q)$-algebra isomorphism $\UUij \rightarrow \UUji$, which sends $q_0 \mapsto q_1, q_1 \mapsto q_0, \bt_0 \mapsto \bt_r, \be_i \mapsto \be_{r-i}, \bbf_i \mapsto \bbf_{r-i}, \bk_i \mapsto \bk_{r-i}$, for $1\leq i\leq r$.
In particular, the Serre relations for $\bt_0$ are as follows:
\begin{align}\label{eq:Serreij}
&\tji_0^2 \eji_{1} + \eji_{1} \tji_0^2  = (q + q^{-1}) \tji_0 \eji_{1} \tji_0 + q_0\inv q_1  \eji_{1},
\quad
\tji_0^2 \fji_{1} + \fji_{1} \tji_0^2  = (q + q^{-1}) \tji_0 \fji_{1} \tji_0 + q_0\inv q_1 \fji_{1}.
\end{align}
We refer to  \S\ref{sec:ji} for the rest of relations of $\UUji$.

The following proposition is a variant of Proposition~\ref{prop:homji} associated to the Dynkin diagram below in Figure~\ref{figure:ij}.

\begin{figure}[ht!]
\caption{Dynkin diagram of type $A^{(1)}_{2r}$ with involution of type $\imath\jmath$.}
   \label{figure:ij}
\begin{tikzpicture}
\matrix [column sep={0.6cm}, row sep={0.5 cm,between origins}, nodes={draw = none,  inner sep = 3pt}]
{
	&\node(U1) [draw, circle, fill=white, scale=0.6, label = 1] {};
	&\node(U3) {$\cdots$};
	&\node(U4)[draw, circle, fill=white, scale=0.6, label =$r-1$] {};
	&\node(U5)[draw, circle, fill=white, scale=0.6, label =$r$] {};
\\
	\node(L)[draw, circle, fill=white, scale=0.6, label =0] {};
	&&&&&
\\
	&\node(L1) [draw, circle, fill=white, scale=0.6, label =below:$2r$] {};
	&\node(L3) {$\cdots$};
	&\node(L4)[draw, circle, fill=white, scale=0.6, label =below:$r+2$] {};
	&\node(L5)[draw, circle, fill=white, scale=0.6, label =below:$r+1$] {};
\\
};
\begin{scope}
\draw (L) -- node  {} (U1);
\draw (U1) -- node  {} (U3);
\draw (U3) -- node  {} (U4);
\draw (U4) -- node  {} (U5);
\draw (U5) -- node  {} (L5);
\draw (L) -- node  {} (L1);
\draw (L1) -- node  {} (L3);
\draw (L3) -- node  {} (L4);
\draw (L4) -- node  {} (L5);
\draw (L) edge [color = blue, loop left, looseness=40, <->, shorten >=4pt, shorten <=4pt] node {} (L);
\draw (L1) edge [color = blue,<->, bend right, shorten >=4pt, shorten <=4pt] node  {} (U1);
\draw (L4) edge [color = blue,<->, bend left, shorten >=4pt, shorten <=4pt] node  {} (U4);
\draw (L5) edge [color = blue,<->, bend left, shorten >=4pt, shorten <=4pt] node  {} (U5);
\end{scope}
\end{tikzpicture}
\end{figure}

\begin{prop}
 \label{prop:homij}
There is an injective $\cR$-algebra homomorphism
$\ijw: \UUij \rightarrow \UUAij$ such that
\bA{
\notag
&\bk_a \mapsto  \bK_a\bK_{-a-1}^{-1},   \quad (1\leq a \le  r)
\\
\notag
&\be_i\mapsto \bE_i+\bF_{-i-1}\bK_i^{-1},
\quad
\bbf_i\mapsto \bE_{-i-1} + \bF_i\bK_{-i-1}^{-1},
\quad (1\leq i \leq r-1)
\\
\notag
&\be_r\mapsto \bE_r+q^{-1}\bF_{-r-1}\bK_r^{-1},
\quad
\bbf_r\mapsto \bE_{-r-1} + q_0q_1^{-1}\bF_r\bK_{-r-1}^{-1},
\\
&  \label{embeding4ij}
\bt_0 \mapsto \bE_0+qq_0\inv q_1 \bF_0 \bK_{0}^{-1} +(q_1-q_0^{-1})/(q-q^{-1}) \bK_{0}^{-1}.
}
\end{prop}

Recalling $\Ld_{n,d}$ from \eqref{def:Ld},
we define
\begin{align*}
\Lambda^{\imath\jmath}_{\nn,d} &= \{ \lambda =(\lambda_0, \lambda_1, \ldots, \lambda_{r+1}) \in \Ld_{n,d}~|~\lambda_0=0\}, \\
\end{align*}
and define the right $\bbH$-module
\begin{align*}
\bTij &= \bigoplus_{\lambda \in \Lambda^{\imath\jmath}_{\nn,d}} x_{\lambda}\bbH.
\end{align*}
Following \cite{FL3Wb}, we define the $\ijw$ variant of the Schur algebra $\bSjj$ as follows:
\begin{align*}
\bSij &= \End_{\bbH}(\oplus_{\lambda \in \Lambda^{\imath\jmath}_{\nn,d}} x_{\lambda}\bbH) =\bigoplus_{\lambda,\mu \in \Lambda^{\imath\jmath}_{\nn,d}} \Hom_{\bbH}(x_{\mu}\bbH,x_{\lambda}\bbH).
\end{align*}

The following is a variant of Lemma~\ref{lem:iso}.
\begin{lem}
We have an isomorphism of $\bbH$-modules:
$\bTij \cong \mathbb{V}_{\nn}^{\otimes d}.$
\end{lem}

Note $\UUij$ acts on $\mathbb{V}_{\nn}^{\otimes d}$ via the embedding $\ijw: \UUij \rightarrow \UUA$; we denote this action by $\Psi^{\imath\jmath}$.
In particular, we give the $\bt_0$-action on $\VV_\nn$ for record in the following:
for $f \not\equiv 0$, we have
\begin{align*}
\bt_0 v_f =
\bc{
v_{f-2} + \frac{q_1-q_0\inv}{q-q\inv} q v_f
&\tif f \equiv 1;
\\
q_0\inv q_1 v_{f+2} + \frac{q_1-q_0\inv}{q-q\inv} q\inv v_f
&\tif f \equiv -1;
\\
\frac{q_1-q_0\inv}{q-q\inv} v_f
&\otw.
}
\end{align*}

The following is a variant of Theorem~\ref{thm:3-par:ji}, and can be proved similarly.

\begin{thm}
 \label{thm:3-par:ij}
We have the following Schur $(\UUij, \bbH)$-duality:
\begin{align*}
\Psi^{\imath\jmath} (\UUij) \simeq  & \End_{\bbH} (\mathbb{V}_{\nn}^{\otimes d}),
\\
& \End_{\UUij} (\mathbb{V}_{\nn}^{\otimes d})  \simeq \bbH^{op}.
\end{align*}
\end{thm}

\begin{rem}
Starting with a natural $\UUAij$-module 
$\VV_{\nn,\frac12}$ with a basis parametrized by $\frac12+\ZZ$ of periodicity $\nn$, we can reformulate the Schur $(\UUji, \bbH)$-duality in Theorem~\ref{thm:3-par:ji} on $\mathbb{V}_{\nn,\frac12}^{\otimes d}$accordingly.
Similarly, starting with a natural $\UUAij$-module 
$\VV_{\nn,0}$ with a basis parametrized by $\ZZ$ of periodicity $\nn$, we can reformulate the Schur $(\UUij, \bbH)$-duality in Theorem~\ref{thm:3-par:ij} on $\mathbb{V}_{\nn,0}^{\otimes d}$ accordingly.
\end{rem}

\subsection{The $\iiw$-Schur duality}

We shall assume $r\ge 2$ in this subsection.
Let $\mathbb{V}_{\eta}$ be the $\cR$-subspace of $\mathbb{V}$ spanned by $v_i$, for $i\in \ZZ$ such that $i \not\equiv 0 \pmod n$ and $i \not\equiv r+1 \pmod n$.  Note $\mathbb{V}_{\eta} =\mathbb{V}_{\nn} \cap {}'\mathbb{V}_{\nn}$  is naturally an $\bbH$-submodule of $\mathbb{V}$, and moreover, it is a direct sum of permutation modules.

We consider the $\cR$-algebra $\UUAiig$ (with an unusual  indexing set of generators). The algebra $\UUAiig$ is generated by $\bE_i, \bF_i$
$(i \in [0,\nn-1]\backslash \{r+1\}),
\bD_a^{\pm1}$ $(a \in [{1},\nn]\backslash \{r+1\});$
here we regard indices $r, r+2$ adjacent. Denote by $\UUAii$ the subalgebra of $\UUAiig$ {generated by $\bE_i, \bF_i, \bK_i$ $(i \in [0,\nn-1]\backslash \{r+1\})$, where $\bK_0=\bD_{-1}\bD_{1}^{-1}, \bK_{r}=\bD_{r}\bD_{r+2}^{-1}, \bK_i=\bD_{i}\bD_{i+1}^{-1}$ ($i\neq 0,r$).}
Then $\mathbb{V}_{\eta}$ is a natural representation of $\UUAiig$, with the action given by: for $i \in [0,\nn-1]\backslash \{r+1\},  a \in [0,\nn]\backslash \{r+1\}$,
\begin{align}
\label{actionUii}
\begin{split}
\bE_{i}v_{j+1} &=
\bc{v_{j}&\tif j\equiv i \neq 0,r;
\\
v_{j-1}&\tif j\equiv i = 0;
\\
v_{j-1}&\tif j-1\equiv i = r;
\\
0&\text{else},}
\\
\bF_{i}v_{j} &=\bc{
v_{j+1} &\tif j\equiv i \neq 0,r;
\\
v_{j+2} &\tif j+1 \equiv i =0;
\\
v_{j+2} &\tif j\equiv i =r;
\\
0 &\text{else},
}
\qquad
\bD_{a}v_{j}=
\bc{q v_{j} &\tif j\equiv a;
\\
v_{j} &\text{else}.
}
\end{split}
\end{align}
Then $\UUAiig$ acts on $\mathbb{V}_{\eta}^{\otimes d}$ via iterated comultiplication.


Define  $\UUii$ to be the $\cR$-algebra generated by $\bt_{0}, \bt_{r},  \be_i, \bbf_i, \bk_i^{\pm1} \; (1 \leq i \leq r-1)$, subject to the relation $\bt_0 \bt_r =\bt_r \bt_0$, and other defining relations which can be found in the algebras $\UUji$ and $\UUij$. (The relations would be different in case $r=1$ as $\bt_{0}$ and $\bt_{r}$ no longer commute.)

\begin{prop}
 \label{prop:homii}
There is an injective $\cR$-algebra homomorphism
$\iiw: \UUii \rightarrow \UUAii$ defined by
\begin{align*}
&\bk_i \mapsto  \bK_i \bK_{-i-1}^{-1},
\\
&\be_i\mapsto \bE_i+\bF_{-i-1}\bK_i^{-1},
\quad
\bbf_i\mapsto \bE_{-i-1} + \bF_i\bK_{-i-1}^{-1},
\quad (1\leq i \leq r-1)
\\
& \bt_0 \mapsto \bE_0+qq_0\inv q_1 \bF_0 \bK_{0}^{-1} +(q_1-q_0^{-1})/(q-q^{-1}) \bK_{0}^{-1},
\\
& \bt_r \mapsto \bE_r+qq_0q_1^{-1} \bF_r\bK_{r}^{-1}+(1-q_0q_1^{-1})/(q-q^{-1})\bK_{r}^{-1}.
\end{align*}
\end{prop}

\begin{proof}
The proof is similar to the proof for Proposition \ref{prop:homjj}. The subalgebra here is a quantum symmetric pair coideal subalgebra associated with the Dynkin diagram and involution below, and the proposition follows from {\cite[Theorem~ 7.8]{Ko14}}.
\begin{figure}[ht!]
\caption{Dynkin diagram of type $A^{(1)}_{2r-1}$ with involution of type $\imath\imath$.}
   \label{figure:ii}
\begin{tikzpicture}
\matrix [column sep={0.6cm}, row sep={0.5 cm,between origins}, nodes={draw = none,  inner sep = 3pt}]
{
	&\node(U1) [draw, circle, fill=white, scale=0.6, label = 1] {};
	&\node(U2) {$\cdots$};
	&\node(U3)[draw, circle, fill=white, scale=0.6, label =$r-1$] {};
\\
	\node(L)[draw, circle, fill=white, scale=0.6, label =0] {};
	&&&&
	\node(R)[draw, circle, fill=white, scale=0.6, label =$r$] {};
\\
	&\node(L1) [draw, circle, fill=white, scale=0.6, label =below:$2r-1$] {};
	&\node(L2) {$\cdots$};
	&\node(L3)[draw, circle, fill=white, scale=0.6, label =below:$r+1$] {};
\\
};
\begin{scope}
\draw (L) -- node  {} (U1);
\draw (U1) -- node  {} (U2);
\draw (U2) -- node  {} (U3);
\draw (U3) -- node  {} (R);
\draw (L) -- node  {} (L1);
\draw (L1) -- node  {} (L2);
\draw (L2) -- node  {} (L3);
\draw (L3) -- node  {} (R);
\draw (L) edge [color = blue, loop left, looseness=40, <->, shorten >=4pt, shorten <=4pt] node {} (L);
\draw (R) edge [color = blue,loop right, looseness=40, <->, shorten >=4pt, shorten <=4pt] node {} (R);
\draw (L1) edge [color = blue,<->, bend right, shorten >=4pt, shorten <=4pt] node  {} (U1);
\draw (L3) edge [color = blue,<->, bend left, shorten >=4pt, shorten <=4pt] node  {} (U3);
\end{scope}
\end{tikzpicture}
\end{figure}
\end{proof}

Recalling $\Ld_{n,d}$ from \eqref{def:Ld},
we define
\begin{align*}
\Lambda^{\imath\imath}_{{\eta},d} &= \{ \lambda =(\lambda_0, \lambda_1, \ldots, \lambda_{r+1}) \in \Ld_{n,d}~|~\lambda_0= \lambda_{r+1}=0\} \; (= \Lambda^{\jmath\imath}_{\nn,d} \cap \Lambda^{{\imath\jmath}}_{{\nn},d}).
\end{align*}

Define the right $\bbH$-module
\begin{align*}
\bTii &= \bigoplus_{\lambda \in \Lambda^{\imath\imath}_{{\eta},d}} x_{\lambda}\bbH.
\end{align*}
Following \cite{FL3Wb}, we define the $\iiw$-variant of the Schur algebra $\bSjj$ as follows:
\begin{align*}
\bSii &= \End_{\bbH}(\oplus_{\lambda \in \Lambda^{\imath\imath}_{\eta,d}} x_{\lambda}\bbH) =\bigoplus_{\lambda,\mu \in \Lambda^{\imath\imath}_{\eta,d}} \Hom_{\bbH}(x_{\mu}\bbH,x_{\lambda}\bbH).
\end{align*}

The following is a variant of Lemma~\ref{lem:iso}.
\begin{lem}
We have an isomorphism of $\bbH$-modules:
$\bTii \cong \mathbb{V}_{\eta}^{\otimes d}.$
\end{lem}

Note $\UUii$ acts on $\mathbb{V}_{\eta}^{\otimes d}$ via the embedding $\iiw: \UUii \rightarrow \UUAii$; we denote this action by $\Psi^{\imath\imath}$. The following theorem can be established similarly as for Theorem~\ref{thm:3-par:ji}.

\begin{thm}
 \label{thm:3-par:ii}
Let $r\ge d \ge 2$. We have the following Schur $(\UUii, \bbH)$-duality:
\begin{align*}
\Psi^{\imath\imath}(\UUii) \simeq & \End_{\bbH} (\mathbb{V}_{\eta}^{\otimes d}),
\\
& \End_{\UUii} (\mathbb{V}_{\eta}^{\otimes d})  \simeq \bbH^{op}.
\end{align*}
\end{thm}
\proof
The proof is similar to the previous counterparts except that we need the following formulas for the images of $\bt_0, \bt_r$ which follow from a direct computation:
\begin{equation*}
\Psi^{\ii}(\bt_0) \in \sum_{\ld \in \Ld^{\ii}_{\eta, d}}\cR \phi^{s_0}_{\ld, \ld},
\quad
\Psi^{\ii}(\bt_r) \in \sum_{\ld \in \Ld^{\ii}_{\eta, d}}\cR \phi^{s_d}_{\ld, \ld}.
\end{equation*}
The theorem is proved.
\endproof


\end{document}